\documentclass[12pt]{amsart}

\usepackage{times}
\usepackage{amsmath,amsfonts,amssymb,amsopn,amscd,amsthm}
\usepackage{graphicx,xspace}
\usepackage{epsfig}
\usepackage{enumerate} 
\usepackage{enumitem}
\usepackage{xcolor}
\usepackage[all]{xypic}

\usepackage[T1]{fontenc}
\usepackage[utf8]{inputenc}

\usepackage{hyperref}
\usepackage{psfrag}
\usepackage{bm}
\usepackage{youngtab}
\usepackage{tikz}

\theoremstyle{plain}
\newtheorem{lemma}{Lemma}[section]

\theoremstyle{remark}

\newtheorem{definition}[lemma]{Definition}

\author[O. Tout]{Omar Tout}
\address{LaBRI, Universit\'e de Bordeaux, 351 cours de la Lib\'eration, 33 400
Talence, France}
\email{omar.tout@labri.fr}

\title[Polynomiality property of the structure coefficients of double-class algebras]{A general framework for the polynomiality property of the structure coefficients of double-class algebras}

\usetikzlibrary{snakes}
\usepackage[all]{xy}
\usepackage{varioref}
\theoremstyle{plain}
\newtheorem{theoreme}{Theorem}[section]
\newtheorem{prop}[theoreme]{Proposition}
\newtheorem{ex}{Example}[section]
\newtheorem{c-ex}{Counter-example}[section]
\newtheorem{que}{Question}[section]
\newtheorem{cor}[theoreme]{Corollary}
\newtheorem{lem}[theoreme]{Lemma}

\theoremstyle{definition}
\theoremstyle{remark}
\newtheorem*{rem}{Remark}
\newtheorem{obs}[subsection]{Observation}
\newtheorem*{notation}{Notation}

\date{}

\DeclareUnicodeCharacter{03C7}{\chi}
\DeclareUnicodeCharacter{B2}{^{2}}

\DeclareMathOperator{\ct}{ct}
\DeclareMathOperator{\SC}{SC}
\DeclareMathOperator{\NCSym}{NCSym}

\DeclareMathOperator{\diag}{diag}

\newcommand{\Hecke}{$\mathbb{C}[\mathcal{B}_n\setminus \mathcal{S}_{2n}/\mathcal{B}_n]$}
\newcommand{\hecke}{$(\mathcal{S}_{2n}, \mathcal{B}_n)$}
\begin{document}
\maketitle
\begin{abstract}
Take a sequence of couples $(G_n,K_n)_n$, where $G_n$ is a group and $K_n$ is a sub-group of $G_n.$ Under some conditions, we are able to give a formula that shows the form of the structure coefficients that appear in the product of double-classes of $K_n$ in $G_n.$ We show how this can give us a similar result for the structure coefficients of the centers of group algebras. 

These formulas allow us to re-obtain the polynomiality property of the structure coefficients in the cases of the center of the symmetric group algebra 
and the Hecke algebra of the pair $(\mathcal{S}_{2n},\mathcal{B}_{n}).$ We also give a new polynomiality property for the structure coefficients of the center of the hyperoctahedral group algebra and the double-class algebra $\mathbb{C}[diag(\mathcal{S}_{n-1})\setminus \mathcal{S}_n\times \mathcal{S}^{opp}_{n-1}/ diag(\mathcal{S}_{n-1})].$
\end{abstract}

\section{Introduction}

The structure coefficients define the product of basis elements of a finite dimensional algebra. Compute these coefficients is essential because it allows to compute all the products in the considered algebra. However, giving an explicit formula for the structure coefficients is a difficult problem even for specific algebras. 

\subsection{Background}

The structure coefficients of centers of finite group algebras have been the most studied in the literature. By Frobenius theorem, see \cite[Lemma 3.3]{JaVi90} and the appendix of Zagier in \cite{lando2004graphs}, these coefficients are expressed in terms of irreducible characters. This links the study of these structure coefficients to representation theory of finite groups.


Other important cases of algebras such that double-class algebras were also studied in the literature. The author has recently established, in his paper \cite{toutAFrobe}, a theorem  similar to that of Frobenius. This theorem expresses the structure coefficients of double-class algebras of Gelfand pairs in terms of zonal spherical functions. 



The case of the center of the symmetric group algebra is particularly interesting and many authors have studied it in details. To compute the structure coefficients of the center of the symmetric group algebra one should understand the cycle-type of product of permutations, see for example the papers \cite{bertram1980decomposing}, \cite{boccara1980nombre}, \cite{Stanley1981255}, \cite{walkup1979many}, \cite{GoupilSchaefferStructureCoef} and \cite{JaVi90} which deal with particular cases of these coefficients. Many authors, see \cite{Stanley1981255}, \cite{JaVi90}, \cite{jackson1987counting}, \cite{GoupilSchaefferStructureCoef}, used the irreducible characters to compute these coefficients but these results were also difficult to be found. Despite all efforts done to compute the structure coefficients of the center of the symmetric group algebra, there is no general formula to compute these coefficients and this problem is still open. 


The Hecke algebra of the pair $(\mathcal{S}_{2n}, \mathcal{B}_n),$ where $\mathcal{B}_n$ is the hyperoctahedral sub-group of $\mathcal{S}_{2n},$ was introduced by James in 1961 in \cite{James1961}. It has a long list of properties similar to that of the center of the symmetric group algebra, see \cite{Aker20122465} and \cite{toutejc}. In particular, they both have a basis indexed by the set of partitions of $n.$ In addition, the structure coefficients, associated to these bases, are related to symmetric functions and graphs embedded into surfaces.

By Frobenius formula, see \cite{sagan2001symmetric}, the irreducible characters of the symmetric group appear in the expansion of Schur functions in terms of power functions. This relates structure coefficients of the center of the symmetric group algebra to the theory of symmetric functions.

The pair $(\mathcal{S}_{2n}, \mathcal{B}_n)$ is a Gelfand pair, see \cite[Section VII.2]{McDo}. Structure coefficients of the Hecke algebra of the pair $(\mathcal{S}_{2n}, \mathcal{B}_n),$ are also related to the theory of symmetric functions since the zonal spherical functions of the pair $(\mathcal{S}_{2n}, \mathcal{B}_n)$ appear in the expansion of zonal polynomials in terms of power functions. The zonal polynomials are specialisations of Jack polynomials, defined by Jack in \cite{jack1970class} and \cite{jack1972xxv}. They form a basis for the algebra of symmetric functions. 

In 1975, Cori proved in his thesis, see \cite{CoriHypermaps}, that the structure coefficients of the center of the symmetric group algebra count the number of graphs embedded into oriented surfaces. This result can also be found in the book \cite{lando2004graphs} of Lando and Zvonkin and the paper \cite{jackson1990character} of Jackson and Visentin.

The Hecke algebra of the pair $(\mathcal{S}_{2n}, \mathcal{B}_n)$ has a similar combinatorial interpretation. Its structure coefficients count the number of graphs embedded into non-oriented surfaces, as established by Goulden and Jackson in \cite{GouldenJacksonLocallyOrientedMaps}. 


The relation between structure coefficients and graphs is not limited to the cases of the center of the symmetric group algebra and the Hecke algebra of the pair $(\mathcal{S}_{2n}, \mathcal{B}_n).$ For example, the structure coefficients of the Hecke algebra of the pair $(\mathcal{S}_n\times \mathcal{S}_{n-1}^{opp},\diag(\mathcal{S}_{n-1}))$ are related to particular graphs called dipoles (see Jackson and Sloss \cite{Jackson20121856}). For more details about this algebra, the reader is invited to see the papers \cite{brender1976spherical}, \cite{strahov2007generalized} and \cite{jackson2012character}. In \cite{strahov2007generalized}, Strahov shows that the zonal spherical functions of this pair generalise some properties of irreducible characters in the case of the center of the symmetric group algebra. 

\bigskip
As already stated, the computation, by a direct way or by using the irreducible characters, of the structure coefficients of the center of the symmetric group algebra is difficult. Computing the structure coefficients is even harder in the case of the Hecke algebra of the pair $(\mathcal{S}_{2n}, \mathcal{B}_n),$ see \cite{bernardi2011counting}, \cite{morales2011bijective} and \cite{vassilieva2012explicit}.

In 1959, Farahat and Higman proved in \cite{FaharatHigman1959} a polynomiality property in $n$ for the structure coefficients of the center of the symmetric group algebra. In 1999, Ivanov and Kerov gave in \cite{Ivanov1999} a combinatorial proof to Farahat and Higman's theorem. They introduce in this paper combinatorial objects which they call partial permutations.

In \cite{toutejc}, we gave a combinatorial proof for the polynomiality property of the structure coefficients of the Hecke algebra of the pair $(\mathcal{S}_{2n}, \mathcal{B}_n)$ similar to that obtained by Ivanov and Kerov in the case of the center of the symmetric group algebra. This polynomiality property was also found by Aker and Can, see \cite{Aker20122465}, and by Do{\l}{\c e}ga and Féray, see \cite{2014arXiv1402.4615D}.


Recently, Méliot has found, see \cite{meliot2013partial}, a polynomiality property, similar to that of Farahat and Higman, for the structure coefficients of the center of the group algebra of invertible matrices with coefficients in a finite field. 
It is worth mentioning that Méliot has already proved in \cite{meliot2010products} a polynomiality property for the structure coefficients of the Iwahori–Hecke algebra of the symmetric group conjectured by Francis and Wang in \cite{francis1992centers}.




\subsection{Our results}

In this paper we study the structure coefficients of a large family of double-class algebras, which include the centers of group algebras. Particularly, we are interested in the dependence in $n$ of these coefficients in the case of a sequence of double-class algebras.

Our main result implies polynomiality properties for the center of the symmetric group algebra and the Hecke algebra of the pair $(\mathcal{S}_{2n}, \mathcal{B}_n).$ We also give a polynomiality property for the structure coefficients of the double-class algebra of the pair $(\mathcal{S}_n\times \mathcal{S}_{n-1}^{opp},\diag(\mathcal{S}_{n-1})).$ These latter have a combinatorial interpretation using special graphs, see \cite{Jackson20121856}.


Another application of our generalisation is the case of the center of the hyperoctahedral group algebra. The reader is invited to see the papers \cite{geissinger1978representations} and \cite{stembridge1992projective} for the details of this algebra. Once again, our generalisation implies a polynomiality property for the structure coefficients in this case.


\subsection{Further work}

Unfortunately, our general framework does not contain the case of the center of the group algebra of invertible matrices with coefficients in a finite field and Méliot's result in this case. 

Another important case is that of the super-classes of uni-triangular matrices groups. Recently, these objects had been the subject of an intense research work, see \cite{andre2013supercharacters}, \cite{UnitriangulargroupAndre}, \cite{diaconis2008supercharacters} and \cite{yan2001representation}, and a polynomiality property for the structure coefficients in this case will be of value. In fact, the structure coefficients of super-classes of uni-triangular matrices groups can be viewed as the structure coefficients of a particular double-class algebra. However, unfortunately our general framework does not contain this case also. One of the interesting things to do, in a future work about our generalisation, is to see whether or not our general framework can be modified in order to contain these two models --Méliot's one and the super-classes of uni-triangular groups--.

The author thinks that the list of Gelfand pairs given by Strahov in \cite[Section 1.3]{strahov2007generalized} may be of particular interest to our generalisation. In fact, the majority of pairs given in that list are formed by symmetric groups (and deformations of symmetric groups). When a sequence of pairs is formed by symmetric groups, there is a good chance that this sequence enters in our generalisation, see Section \ref{sec:appl_double_classe} for more details.

\section{General framework: Definitions and main theorem}
In this section, we present our general framework for the polynomiality property of the structure coefficients and we give all necessary definitions to present our main theorems. Our general framework is about the double-class algebras.

\subsection{Hypotheses and definitions}\label{sec:hypo_defi} Let $G$ be a group and $K$ a subgroup of $G.$ A \textit{double-class} of $K$ in $G$ is a set $KxK$ for an element $x$ of $G.$ The set of double-classes of $K$ in $G$ is denoted by $K\setminus G/K.$ The \textit{double-class algebra} of $K$ in $G,$ denoted by $\mathbb{C}[K\setminus G/K],$ is the algebra over $\mathbb{C}$ with basis the (formal) sums of the elements of the double-classes of $K$ in $G.$

Let $(G_n,K_n)_n$ be a sequence where $G_n$ is a group and $K_n$ is a sub-group of $G_n$ for each $n.$ We also suppose that $G_n\subseteq G_{n+1}$ and that $K_n\subseteq K_{n+1}$ for each $n.$ For each $n,$ we consider the set $K_{n}\setminus G_{n}/K_{n}$ of double-classes of $K_{n}$ in $G_{n}.$ We require that if $x \in G_n,$ then the intersection of the $K_{n+1}$-double-class of $x$ with $G_n$ is the $K_n$-double-class of $x$: formally,
\begin{enumerate}
\item[H.0]\label{hyp0} $K_{n+1}xK_{n+1}\cap G_n=K_{n}xK_{n}$ for each $x\in G_n.$
\end{enumerate} 

For an element $x$ of $G_n,$ we denote by $\overline{x}^{n}$ the $K_n$-double-class $K_nxK_n.$ We recall that $y\in \overline{x}^{n}$ if and only if there exists two elements $k,k'\in K_{n}$ such that $x=kyk'.$

\begin{notation}
If $X$ is a finite set, we denote by ${\bf X}$ the formal sum of its elements,
$$\textbf{X}=\sum_{x\in X}x.$$
\end{notation}

For a triple elements $(x_1,x_2,x_3)$ of $G_n,$ we define $c_{1,2}^3(n)$ to be the coefficient of $\overline{\bf x_3}^{n}$ in the product $\overline{\bf x_1}^{n}\cdot \overline{\bf x_2}^{n}.$

In our framework, we suppose that for each $k\leq n,$ there exists a sub-group $K_n^k$ of $K_n$ which satisfies the following hypotheses :

\begin{enumerate}
\item[H.1]\label{hyp1} $K_n^k$ is isomorphic as a group to $K_{n-k}$.
\item[H.2]\label{hyp2} If $x\in K_k$ and $y\in K_n^k$, then we have : $x\cdot y=y\cdot x.$
\item[H.3]\label{hyp3} $K_{n+1}^k \cap K_n=K_n^k$ if $k\leq n.$
\end{enumerate}

In these hypotheses, there is no conditions on double-classes. To present the hypotheses which involve double-classes, we define a function $\mathrm{k}$ as follows :
\begin{equation*}
\mathrm{k}(X):=\min_{k\atop{X\cap K_k\neq \emptyset}}k,
\end{equation*}
for any subset $X$ of $\cup_{n\geq 1}K_n.$ This definition will be crucial for us to present as well as to prove our main results.

\begin{definition}
Let $y\in K_n$, we say that $y$ is $(k_1,k_2)$-minimal if $y\in K_m$ where $m=\mathrm{k}(K_n^{k_1}yK_n^{k_2})$.
\end{definition}

This will also be an important definition for the following. We present here the necessary conditions on double-classes in our general framework. The list of conditions is the following :

\begin{enumerate}
\item[H.4]\label{hyp5} $m_{k_1,k_2}(x):=\mathrm{k}(K_n^{k_1}xK_n^{k_2})\leq k_1+k_2$ for any $x\in K_n.$
\item[H.5]\label{hyp6} $yK_{n}^{k_1}y^{-1}\cap K_{n}^{k_2}=K_{n}^{m_{k_1,k_2}(y)},$ if $y$ is $(k_1,k_2)$-minimal.
\end{enumerate}

\begin{rem}
The hypotheses required in our general framework are inspired from particular cases already studied in relation with a polynomiality property of structure coefficients. It is remarkable that the sequence $G_n$ of groups does not appear in hypotheses H.1 to H.5: these hypotheses only involve the sequence $K_n$ of sub-groups. This will be practical for applications in the next sections of this paper. The fourth hypothesis H.4 is the most important among these hypotheses (the reader can have a look to the applications we give at the end of this paper to better understand this hypothesis). The only hypothesis which depends on the sequence $(G_n,K_n)_n,$ and not just on the sequence $(K_n)_n,$ is H.0. This hypothesis ensures the independence on $n$ of the intersection of $G_{n_0}$ with the $K_n$-double-classes of its elements for a fixed $n_0$ and a sufficiently big $n.$ It is easily verified in the particular cases of sequences $(G_n,K_n)_n$ which we will consider in the application sections \ref{sec:appl_double_classe} and \ref{sec:appl_centr}.
\end{rem}
\begin{rem}
With hypotheses H.0 to H.5, we present conditions which imply a polynomiality property for the structure coefficients. We have found them using the already known results of polynomiality, especially those of $Z(\mathbb{C}[\mathcal{S}_n])$ and $\mathbb{C}[\mathcal{B}_n\setminus \mathcal{S}_{2n}/\mathcal{B}_n].$ It is important to see whether or not the list of conditions which we give is "minimal" (that means that there isn't any hypothesis resulting from a set of other hypotheses in the list). We think that our list of hypotheses H.0 to H.5 is minimal. However, since we are interested in double-classes in our approach, we should point out that hypothesis H.3 is equivalent to the following hypothesis:
\begin{enumerate}
\item[H'.3]\label{hyp4} For any element $z\in K_n$, we have $K_{n+1}^{k_1}zK_{n+1}^{k_2}\cap K_n=K_{n}^{k_1}zK_{n}^{k_2}.$
\end{enumerate}
We give the proof of this equivalence in Observation \ref{obs:Equiv_H3_H4} which follows this remark. H'.3 seems to be more appropriate to our approach since it is presented as an hypothesis on double-classes. We decided to put H.3 in our list instead of H'.3 because it is easier to verify for our applications but we will not forget the usefulness of H'.3 in the next sections especially to prove our main results in the general case. Hence, we will use both hypotheses in the coming sections.
\end{rem}
\begin{obs}\label{obs:Equiv_H3_H4}
Hypothesis H.3 is obtained from H'.3 while considering the particular case where $k_1=k_2=k$ and $z$ is the neutral element. In the opposite direction, if H.3 is verified, for each $z\in K_n$ we have: $K_{n+1}^{k_1}z \cap K_n=K_n^{k_1}z$ for each $k_1\leq n.$ We also have, $K_{n+1}^{k_2} \cap K_n=K_n^{k_2}$ for each $k_2\leq n.$ If we multiply these two equations we get:
\begin{equation*}
\big( K_{n+1}^{k_1}z \cap K_n \big) \big( K_{n+1}^{k_2} \cap K_n \big ) = K_n^{k_1}zK_n^{k_2}.
\end{equation*}
If we develop the left hand side of this equation we get:
\begin{equation*}
K_{n+1}^{k_1}zK_{n+1}^{k_2} \cap K_{n+1}^{k_1}zK_n \cap K_nK_{n+1}^{k_2} \cap K_n.
\end{equation*}
This is equal to $K_{n+1}^{k_1}zK_{n+1}^{k_2} \cap K_n$ since $K_{n+1}^{k_1}zK_n$ and $K_nK_{n+1}^{k_2}$ both contain $K_n.$ Thus, for each $z\in K_n$ we have $K_{n+1}^{k_1}zK_{n+1}^{k_2}\cap K_n=K_{n}^{k_1}zK_{n}^{k_2}$ for any $k_1, k_2\leq n,$ which is hypothesis H'.3. Therefore hypotheses H.3 and H'.3 are equivalent.
\end{obs}
\begin{obs}\label{obs:H_4_ind_n}
According to H'.3, the minimality does not depend on $n.$ In fact, let us fix three integers $n_0,$ $k_1$ and $k_2,$ and an element $z\in K_{n_0},$ and suppose that $\mathrm{k}(K_{n_0}^{k_1}zK_{n_0}^{k_2})$ is an integer $k_{n_0,k_1,k_2}.$ By H.4,  $k_{n_0,k_1,k_2}\leq k_1+k_2.$ For any $n\geq n_0+1,$ we have $$K_{n}^{k_1}zK_n^{k_2}\cap K_{k_{n_0,k_1,k_2}}=K_{n_0}^{k_1}zK_{n_0}^{k_2}\cap K_{k_{n_0,k_1,k_2}},$$ using H'.3. Thus, for each $n\geq n_0,$ we have $\mathrm{k}(K_{n}^{k_1}zK_{n}^{k_2})=\mathrm{k}(K_{n_0}^{k_1}zK_{n_0}^{k_2}).$ This proves that $m_{k_1,k_2}(x)$ does not depend on $n.$ We will use this observation in the proof of Theorem \ref{mini_th}.
\end{obs}

\subsection{Main theorems}

Our main result in this paper is Theorem \ref{main_th} presented below. It gives us the general form of the structure coefficients of the double-class algebra under the conditions given in Section \ref{sec:hypo_defi}. The polynomiality property for the structure coefficients for specific algebras can be obtained directly by using this theorem. In addition, this theorem gives not only the polynomiality property of structure coefficients, for specific algebra, but also the exact values of these coefficients. 

\begin{theoreme}\label{main_th}
Let $(G_n,K_n)_n$, be a sequence of pairs, where $G_n$ is a group and $K_n$ is a sub-group of $G_n$ for each $n,$ satisfying hypotheses H.0 to H.5. For a fixed integer $n_0$ and three elements $x_1$, $x_2$ and $x_3$ of $G_{n_0},$ we denote by $k_1$ (resp. $k_2$, $k_3$) the integer $\mathrm{k}(\bar{x_1}^{n_0})$ (resp. $\mathrm{k}(\bar{x_2}^{n_0})$, $\mathrm{k}(\bar{x_3}^{n_0})$). The structure coefficient $c_{1,2}^3(n_0)$ of $\bar{{\bf x_3}}^{n_0}$ in the expansion of the product $\bar{{\bf x_1}}^{n_0}\bar{{\bf x_2}}^{n_0}$ is given by the following formula :
\begin{eqnarray}\label{eq:forme_coef_stru}
c_{1,2}^3(n_0)=&&\frac{|\bar{x_1}^{n_0}||\bar{x_2}^{n_0}||K_{n_0-k_1}||K_{n_0-k_2}|}{|K_{n_0}||\bar{x_3}^{n_0}|} \nonumber \\
&&\sum_{ \max(k_1,k_2,k_3)\leq k\leq \min(k_1+k_2,n_0), x\in G_k, \atop{\text{ $x_1^{-1}xx_2^{-1}\in K_k$ and is $(k_1,k_2)$-minimal} \atop{\bar{x}^{n_0}=\bar{x_3}^{n_0}}}}\frac{1}{|K_{n_0-k}||K_{n_0}^{k_1}x_1^{-1}xx_2^{-1}K_{n_0}^{k_2}\cap K_{m_{k_1,k_2}(x_1^{-1}xx_2^{-1})}|}.
\end{eqnarray}
\end{theoreme}
\begin{proof}
See Section \ref{sec:preuve_th_princ} dedicated to the proof of this theorem.
\end{proof}

Applications of Theorem \ref{main_th} are given in Sections \ref{sec:alg_de_hecke} and \ref{alg_double_classe_diag(S_n-1)}. This theorem gives us a formula for the structure coefficients but the size of the set $K_n^{k_1}x_1^{-1}xx_2^{-1}K_n^{k_2}\cap K_{m_{k_1,k_2}(x_1^{-1}xx_2^{-1})}$ is not easily known in the general case (contrary to that of $K_{n-k}$). In addition, the sum index in the equation \eqref{eq:forme_coef_stru} is quite complicated. We can make this theorem easier to use if we are interested in the polynomiality property for the structure coefficients (and not in their exact values). To this purpose we give a second theorem below.

\begin{theoreme}\label{mini_th} Let $n_1$ be a fixed integer and let $x_1,x_2$ and $x_3$ be three elements of $G_{n_1}.$ We denote by $k_1$ (resp. $k_2$ and $k_3$) the integer $\mathrm{k}(\bar{x_1}^n)$ (resp. $\mathrm{k}(\bar{x_2}^n)$ and $\mathrm{k}(\bar{x_3}^n)$). For any $n$ sufficiently big, there exists non-negative rational numbers $a_{1,2}^{3}(k)$ for any $k_3\leq k\leq k_1+k_2$ independent of $n$ such that the structure coefficient $c_{1,2}^3(n)$ can be written as follows :

\begin{equation}
c_{1,2}^3(n)=\frac{|\bar{x_1}^n||\bar{x_2}^n||K_{n-k_1}||K_{n-k_2}|}{|K_n||\bar{x_3}^n|}\sum_{ \max(k_1,k_2,k_3)\leq k\leq k_1+k_2}\frac{a_{1,2}^{3}(k)}{|K_{n-k}|}.
\end{equation}

\end{theoreme}

\begin{proof} We start the proof by applying Theorem \ref{main_th} to $x_1,x_2,x_3$ for any integer $n_0\geq n_1.$ We can do so because the inclusion $G_n\subset G_{n+1}$ allows us to look at $x_1,x_2,x_3$ as elements of $G_{n_0}.$ According to hypothesis H. 0, the integers $k_1$, $k_2$ and $k_3$ are independent of $n_0$ if $n_0$ is big enough, and so is the size of
$K_{n_0}^{k_1}x_1^{-1}xx_2^{-1}K_{n_0}^{k_2}\cap K_{m_{k_1,k_2}(x_1^{-1}xx_2^{-1})}$ (by H'.3 and Observation \ref{obs:H_4_ind_n} implies the independence on $n_0$ of $m_{k_1,k_2}(x_1^{-1}xx_2^{-1})$). It remains to us to remark that hypothesis H.0, on the double-classes of $K_n$ in $G_n,$ with Observation \ref{obs:H_4_ind_n} ensure us that the sum index in Theorem \ref{main_th} does not depend on $n$ if $n$ is sufficiently big. Theorem \ref{mini_th} is thus a consequence of Theorem \ref{main_th}.
\end{proof}


Theorem \ref{main_th} allows us to obtain the exact values of $c_{1,2}^3(n)$ in special cases of triple $(x_1,x_2,x_3)$ (in particular algebras) while Theorem \ref{mini_th} is more adapted to give a polynomiality property to these coefficients for a general triple $(x_1,x_2,x_3).$

\section{The partial elements algebra}
We start this section by giving the definition of partial elements. To give the definition, recall that if $K$ is a subgroup of a group $G,$ a \textit{left} (resp. right) \textit{class} of $K$ in $G$ is a set $xK$ (resp. $Kx$) for a certain element $x$ of $G.$

\begin{definition}
A \textit{partial element} of $G_n$ is a triple $(C,(x;k),C')$, where $k$ is an integer between $1$ and $n,$ $C$ (resp. $C'$) is a left (resp. right) class of $K_n^k$ in $K_n$ and $x\in G_k.$
\end{definition}

In this paper, we use $\mathcal{S}_n$ to denote the symmetric group on $[n]:=\lbrace 1,\cdots,n\rbrace$ and $\mathcal{B}_n$ to denote the hyperoctahedral subgroup of $\mathcal{S}_{2n}.$ An element of $\mathcal{B}_n$ is a permutation of $2n$ which sends every pair $\lbrace 2k-1,2k\rbrace$ where $1\leq k \leq n$ to another pair with the same form. 

For example, $(\mathcal{B}_3^2, ((1\,\, 4\,\,3)(2); 2),\mathcal{B}_3^2)$ is a partial element of $\mathcal{S}_{6}$ (associated to the sequence of pairs $(\mathcal{S}_{2n},\mathcal{B}_n)$) where $\mathcal{B}_3^2$ (see Section \ref{sec:cond_group_hyp} for an explicit definition) is a group isomorphic to $\mathcal{B}_1.$

Partial elements are the principal objects that we use to prove our results. The reader should be informed that while the notation "partial elements" suggests that these elements are generalisations of partial permutations in \cite{Ivanov1999}, partial bijections in \cite{toutejc} and partial isomorphisms in \cite{meliot2013partial}, this is not actually the case. In fact, it is sufficient to remark that, in general, the number of partial elements of $G_n$ is :
$$\sum_{k=1}^n\left(\frac{|K_n|}{|K_{n-k}|}\right)^2|G_k|,$$
which does not coincide with the number of partial bijections of $[2n]$ in the case of the sequence of pairs $(\mathcal{S}_{2n},\mathcal{B}_n).$

\begin{definition}\label{def:_prod_ele_part}
Let $pe_1=(C_1,(x_1;k_1),C'_1)$ and $pe_2=(C_2,(x_2;k_2),C'_2)$ be two partial elements of $G_n.$ We define the product $pe_1\cdot pe_2$ as follows :
$$pe_1\cdot pe_2:=\frac{1}{n^{k_1}_m n^{k_2}_m |C'_1C_2\cap K_{\mathrm{k}(C'_1C_2)}|}\sum_{h\in C'_1C_2 \atop{h\text{ $(k_1,k_2)$-minimal}}}\sum_{i=1}^{n^{k_1}_m}\sum_{j=1}^{n^{k_2}_m}(C_1^i,(x_1hx_2;m),{C'}_2^j),$$
where $m=\max(k_1,k_2,\mathrm{k}(C'_1C_2))$, $n^{k_2}_m=\frac{|K_n^{k_2}|}{|K_n^{m}|}$, $n^{k_1}_m=\frac{|K_n^{k_1}|}{|K_n^{m}|}$ and the classes $C_1^{i}$ (resp. ${C'}_{2}^j$) are defined by the following equations : 
\begin{equation}\label{eq:dec_classe}
C_1=\bigsqcup_{i=1}^{n^{k_1}_m}C_1^j~~~~(\text{resp. } C'_2=\bigsqcup_{j=1}^{n^{k_2}_m}{C'}_2^j).\end{equation}
\end{definition}

\begin{rem}
Since the number $m$ given in Definition \ref{def:_prod_ele_part} is at least $\max(k_1,k_2),$ $K_n^m$ is a sub-group of $K_n^{k_1}$ and $K_n^{k_2}.$ Equation \eqref{eq:dec_classe} is the formal writing of the fact that the left (resp. right) classes of $K_n^{k_1}$ (resp. $K_n^{k_2}$) are disjoint unions of left (resp. right) classes of $K_n^m.$
\end{rem}

It is natural to see whether this product is associative or not. The product between partial bijections is associative and this allowed us to build in \cite{toutejc} a universal algebra which projects onto the Hecke algebra \Hecke \, for each $n.$ Likewise, the products between both partial permutations and partial isomorphisms defined by Ivanov/Kerov and Méliot in \cite{Ivanov1999} and \cite{meliot2013partial} were associative. Universal algebras were also presented in both papers.

The proof of associativity was difficult in both \cite{toutejc} and \cite{meliot2013partial}. We decided to avoid the associativity question between partial elements since we do not need this property to prove Theorem \ref{main_th}. 

We denote by $PE_n$ the set of partial elements of $G_n$. The set $K_n\times K_n$ acts on $PE_n$ by the following action :
$$(a,b)\cdot (C,(x;k),C')=(aC,(x;k),C'b^{-1}).$$
This defines a group action because :
\begin{eqnarray*}
(a_1,b_1)\cdot (a_2,b_2)\cdot (C,(x;k),C')&=&(a_1,b_1)\cdot (a_2C,(x;k),C'b_2^{-1})\\
&=&(a_1a_2C,(x;k),C'b_2^{-1}b_1^{-1})\\
&=&(a_1a_2,b_1b_2)\cdot (C,(x;k),C').
\end{eqnarray*}
\begin{prop}\label{P}
The action of $K_n\times K_n$ on $PE_n$ is compatible with the product in $PE_n,$ which means that :
$$(a,b)\cdot (pe_1\cdot pe_2)=((a,c)\cdot pe_1) \cdot ((c,b)\cdot pe_2),$$
for any $a,b$ and $c$ in $K_n.$
\end{prop}
\begin{proof}
The quantity $(a,b)\cdot (pe_1\cdot pe_2)$ is equal to :
\begin{equation*}
\frac{1}{n^{k_1}_m n^{k_2}_m |C'_1C_2\cap K_{\mathrm{k}(C'_1C_2)}|}\sum_{h\in C'_1C_2 \atop{h\text{ $(k_1,k_2)$-minimal}}}\sum_{i=1}^{n^{k_1}_m}\sum_{j=1}^{n^{k_2}_m}(aC_1^i,(x_1hx_2;m),{C'}_2^jb^{-1}),
\end{equation*}
which is also the value of $((a,c)\cdot pe_1) \cdot ((c,b)\cdot pe_2).$
\end{proof}
We now consider the set $\mathbb{C}[PE_n]$ of linear combinations of partial elements of $G_n$ with coefficients in $\mathbb{C}.$ We expand the action of $K_n\times K_n$ on $PE_n$ by linearity to $\mathbb{C}[PE_n]$ and we denote by $\mathcal{A}_n$ the set of invariant elements under this action.
\begin{lem}
$\mathcal{A}_n$ is stable by multiplication. Namely, if $\alpha_1$ and $\alpha_2$ are two elements of $\mathcal{A}_n$, then $\alpha_1\cdot \alpha_2$ is also in $\mathcal{A}_n.$
\end{lem}
\begin{proof}
This result is a consequence of Proposition \ref{P}.
\end{proof}
\begin{notation}
For an element $x\in G_n$, we denote the double-class $K_n^{k_1}xK_n^{k_2}$ by $Cl_{k_1,k_2}(x).$ Likewise, we define $Cl_{k_1,*}(x)$ (resp. $Cl_{*,k_2}(x)$) to be the left (resp. right) class $K_n^{k_1}x$ (resp. $xK_n^{k_2}$). Finally, let us consider the following three sets :
\begin{enumerate}
\item $CL_{k_1,k_2}(G_n):=\lbrace Cl_{k_1,k_2}(x), x\in G_n\rbrace;$
\item $CL_{k_1,*}(G_n):=\lbrace Cl_{k_1,*}(x), x\in G_n\rbrace;$
\item $CL_{*,k_2}(G_n):=\lbrace Cl_{*,k_2}(x), x\in G_n\rbrace.$
\end{enumerate}
\end{notation}
\begin{prop}
$\mathcal{A}_n$ is generated by the family $(\mathbf{a}_{(x;k)}(n))_{1\leq k\leq n\atop{x\in G_k}}$ where :
$$\mathbf{a}_{(x;k)}(n)=\sum_{C\in CL_{*,k}(K_n)}\sum_{C'\in CL_{k,*}(K_n)}(C,(x;k),C').$$
\end{prop}
\begin{proof}
Let $\alpha\in \mathcal{A}_n$ i.e. for any pair $(a,b)$ of $K_n\times K_n$, we have $(a,b)\cdot \alpha=\alpha$. Since $\alpha\in \mathbb{C}[PE_n]$, we can write :
\begin{equation*}
\alpha=\sum_{1\leq k\leq n\atop{x\in G_k}}\sum_{C\in CL_{*,k}(K_n)}\sum_{C'\in CL_{k,*}(K_n)}c_{k,x,C,C'}(C,(x;k),C'),
\end{equation*}
where the coefficients $c_{k,x,C,C'}$ are in $\mathbb{C}.$  The condition $(a,b)\cdot \alpha=\alpha$ for any $(a,b)$ in $K_n\times K_n$ gives us the following relation between the coefficients $c_{k,x,C,C'}$: 
$$c_{k,x,aC,C'b^{-1}}=c_{k,x,C,C'}\text{ for any $(a,b)\in K_n\times K_n$.}$$
This means that all the elements with the form $(*,(x;k),*)$ have the same coefficient in the expansion of $\alpha$ which ends the proof. 
\end{proof}
\begin{prop}\label{F} Let $x_1$ and $x_2$ be two elements of $G_{k_1}$ and $G_{k_2}$ respectively where $k_1$ and $k_2$ are two integers less or equal to $n$, then we have :
\begin{equation}\label{E}\mathfrak{a}_{(x_1;k_1)}(n)\cdot \mathfrak{a}_{(x_2;k_2)}(n)=\sum_{\max(k_1,k_2)\leq k\leq \min(k_1+k_2,n) \atop{x\in G_k}}c_{(x_1;k_1),(x_2;k_2)}^{(x;k)}(n)\mathfrak{a}_{(x;k)}(n),
\end{equation}
where $c_{(x_1;k_1),(x_2;k_2)}^{(x;k)}(n)$ is equal to :
$$\left\{
\begin{array}{ll}
  \frac{|K_n||K_{n-k}|}{|K_{n-k_1}||K_{n-k_2}||K_n^{k_1}XK_n^{k_2}\cap K_{m_{k_1,k_2}(X)}|} & \qquad \mathrm{if}\quad X=x_1^{-1}xx_2^{-1}\in K_n \text{ and is $(k_1,k_2)$-minimal,}\\
  0 & \qquad \mathrm{otherwise},\\
 \end{array}
 \right. \\$$
\end{prop}
\begin{proof} 
We fix an element $x\in G_k$ for some $k\leq k_1+k_2$ and two classes $C$ and $C'.$ Set $X=x_1^{-1}xx_2^{-1}.$ Let $\mathcal{C}_{(x_1;k_1),(x_2;k_2)}^{(x;k)}(n)$ be the set of pairs $(pe_1,pe_2)$ such that $pe_i$ appears in the development of $\mathfrak{a}_{(x_i;k_i)}(n)$ for $i=1,2$ and $(C,(x;k),C')$ appears in the development of the product $pe_1\cdot pe_2.$ Then, we can write :
$$c_{(x_1;k_1),(x_2;k_2)}^{(x;k)}(n)=\sum_{(pe_1,pe_2)\in \mathcal{C}_{(x_1;k_1),(x_2;k_2)}^{(x;k)}(n)}\frac{1}{n^{k_1}_k n^{k_2}_k|C'_1C_2\cap K_{C'_1C_2}|}.$$ 

Since $C$ and $C'$ are fixed, they both determine $C_1$ and $C'_2$ for any pair $(pe_1,pe_2)$ in $\mathcal{C}_{(x_1;k_1),(x_2;k_2)}^{(x;k)}(n).$ We denote by $A_{(x_1;k_1),(x_2;k_2)}^{(x;k)}(n)$ the following set :
\begin{eqnarray*}
{A}_{(x_1;k_1),(x_2;k_2)}^{(x;k)}(n)=\lbrace (C'_1,C_2)\in CL_{k_1,*}(K_n) \times CL_{*,k_2}(K_n) \text{ such that }X\in C'_1C_2\cap K_{k(C'_1C_2)}\rbrace.
\end{eqnarray*}
This set is in bijection with $\mathcal{C}_{(x_1;k_1),(x_2;k_2)}^{(x;k)}(n).$ Note that it is empty if $X$ is not in $K_n$ or not $(k_1,k_2)$-minimal. Let us then consider the case where $X$ is both in $K_n$ and $(k_1,k_2)$-minimal. In that case, $c_{(x_1;k_1),(x_2;k_2)}^{(x;k)}(n)$ can be written as a sum over the elements of ${A}_{(x_1;k_1),(x_2;k_2)}^{(x;k)}(n)$ in the following way :
\begin{eqnarray*}c_{(x_1;k_1),(x_2;k_2)}^{(x;k)}(n)&=&\sum_{(C'_1,C_2)\in A_{(x_1;k_1),(x_2;k_2)}^{(x;k)}(n)}\frac{1}{n^{k_1}_k n^{k_2}_k|K_n^{k_1}XK_n^{k_2}\cap K_{m_{k_1,k_2}(X)}|}\\
&=&\frac{|A_{(x_1;k_1),(x_2;k_2)}^{(x;k)}(n)|}{n^{k_1}_k n^{k_2}_k|K_n^{k_1}XK_n^{k_2}\cap K_{m_{k_1,k_2}(X)}|}.
\end{eqnarray*}

We are going to show that the action of $K_n$ on $A_{(x_1;k_1),(x_2;k_2)}^{(x;k)}(n)$ defined as follows is transitive:
$$h\cdot (C'_1,C_2)=(C'_1h,h^{-1}C_2).$$
Let $(A_1,A_2)$ and $(B_1,B_2)$ be two elements of $A_{(x_1;k_1),(x_2;k_2)}^{(x;k)}(n).$ Let $a_1,a_2,b_1$ and $b_2$ be four representative elements in $K_n$ for $A_1,A_2,B_1,$ and $B_2$ respectively. To show that the action is transitive, we have to find an element $h\in K_n$ such that :
\begin{eqnarray*}
a_1h&=&xb_1\\
h^{-1}a_2&=&b_2y,
\end{eqnarray*}
where $x\in K_n^{k_1}$ and $y\in K_n^{k_2}.$ In other words, the set $a_1^{-1}K_n^{k_1}b_1\cap a_2K_n^{k_2}b_2^{-1}$ should not be empty which is equivalent to say that $K_n^{k_1}\cap a_1a_2K_n^{k_2}(b_1b_2)^{-1}$ is not empty. This is true since as $(A_1,A_2)$ and $(B_1,B_2)$ are in $A_{(x_1;k_1),(x_2;k_2)}^{(x;k)}(n),$ $X\in A_1A_2$ and $X\in B_1B_2.$ Then $a_1a_2$ can be written $h_1Xh_2$ and $b_1b_2$ can be written $h'_1Xh'_2$ where $h_1,h'_1\in K_n^{k_1}$ and $h_2,h'_2\in K_n^{k_2}$ and $K_n^{k_1}\cap XK_n^{k_2}X^{-1}$ is not empty by H.5.\\

Thus there is one and only one orbit and we have,
$$|A_{(x_1;k_1),(x_2;k_2)}^{(x;k)}(n)|=|K_n\cdot (K_n^{k_1}X,K_n^{k_2})|=\frac{|K_n|}{|X^{-1}K_n^{k_1}X\cap K_n^{k_2}|},$$
since the stabilizer of $(K_n^{k_1}X,K_n^{k_2})$ is the set of elements $h\in K_n$ such that $h\in X^{-1}K_n^{k_1}X$ and $h\in K_n^{k_2}.$
By H.5, the denominator is equal to $|K_{n-k}|$, thus :
$$c_{(x_1;k_1),(x_2;k_2)}^{(x;k)}(n)=\left\{
\begin{array}{ll}
  \frac{|K_n||K_{n-k}|}{|K_{n-k_1}||K_{n-k_2}||K_n^{k_1}XK_n^{k_2}\cap K_{m_{k_1,k_2}(X)}|} & \qquad \mathrm{if}\quad X \text{ is $(k_1,k_2)$-minimal},\\
  0 & \qquad \mathrm{otherwise} .\\
 \end{array}
 \right. \\$$
\end{proof}

Let $\psi_n:PE_n\rightarrow \mathbb{C}[K_n\setminus G_n/K_n]$ be the function defined by :
$$\psi_n((C,(x;k),C'))=\frac{1}{|C||C'|}\sum_{c\in C, c'\in C'}cxc'.$$
\begin{prop}\label{prop:compa_psi_n}
$\psi_n$ is compatible with the product defined in $PE_n,$ which means $$\psi_n(pe_1\cdot pe_2)=\psi_n(pe_1)\cdot \psi_n(pe_2),$$
for any $pe_1$ and $pe_2$ in $PE_n.$
\end{prop}
\begin{proof} Let $pe_1$ and $pe_2$ be two elements of $PE_n.$ From the definition of the product (see Definition \ref{def:_prod_ele_part}) we have :
 $$\psi_n(pe_1\cdot pe_2)=\frac{1}{n^{k_1}_m n^{k_2}_m |C'_1C_2\cap K_{k(C'_1C_2)}|}\sum_{h}\sum_{i}\sum_{j}\frac{1}{|C_{1}^{i}||C'^{j}_{2}|}\sum_{c_{1}^{i}\in C_{1}^{i},c'^{j}_{2}\in C'^{j}_{2}}c_{1}^{i} x_{1} h x_{2} c'^{j}_{2}.$$
We did not write the sum indexes in the above equation to make it easier to read. After simplification, we obtain :
 $$\psi_n(pe_1\cdot pe_2)=\frac{1}{|K_{n-k_1}||K_{n-k_2}||C'_1C_2\cap K_{k(C'_1C_2)}|}\sum_h\sum_{c_1\in C_1, c'_2\in C'_2}c_1x_1hx_2c'_2.$$
 On the other hand, we have :
 $$\psi_n(pe_1)\cdot \psi_n(pe_2)=\frac{1}{|K_{n-k_1}|^2|K_{n-k_2}|^2}\sum_{c_1,c'_1,c_2,c'_2}c_1x_1c'_1c_2x_2c'_2.$$
 Thus $\psi_n$ is compatible with the product defined in $PE_n$ if we have the following equality :
 \begin{equation}\label{coef}
\sum_{c_1,c'_1,c_2,c'_2}c_1x_1c'_1c_2x_2c'_2=\frac{|K_{n-k_1}||K_{n-k_2}|}{|C'_1C_2\cap K_{k(C'_1C_2)}|}\sum_h\sum_{c_1\in C_1, c'_2\in C'_2}c_1x_1hx_2c'_2.
\end{equation}
Let us consider an element $h$ of the set $C'_1C_2\cap K_{k(C'_1C_2)}$ (the sum index set of $h$ in Equation \eqref{coef}). Fix $c'_1\in C'_1$ and $c_2\in C_2$, there exists two elements $h_1\in K_n^{k_1}$ and $h_2\in K_n^{k_2}$ such that $c'_1c_2=h_1hh_2.$ By H.2, since $x_1\in G_{k_1}$ (resp. $x_2\in G_{k_2})$, we thus have $x_1c'_1c_2x_2=x_1h_1hh_2x_2=h_1x_1hx_2h_2.$ Therefore, we have :
$$\sum_{c_1,c'_2}c_1x_1c'_1c_2x_2c'_2=\sum_{c_1\in C_1, c'_2\in C'_2}c_1h_1x_1hx_2h_2c'_2=\sum_{c_1\in C_1, c'_2\in C'_2}c_1x_1hx_2c'_2.$$
The last equality comes from the fact that $C_1h_1=C_1$ and $h_2{C'}_2={C'}_2.$ Since the right-hand side of the equation does not depend on $c'_1$ and $c_2,$ we get :
$$\sum_{c_1,c'_1,c_2,c'_2}c_1x_1c'_1c_2x_2c'_2=|K_{n-k_1}||K_{n-k_2}|\sum_{c_1\in C_1, c'_2\in C'_2}c_1x_1hx_2c'_2.$$
Recall that $h$ here is a fixed element of $C'_1C_2\cap K_{k(C'_1C_2)}.$ If we take the sum over the elements $h$ in $C'_1C_2\cap K_{k(C'_1C_2)}$ and since the left-hand side of the above equation does not depend on $h$, we get Equation \eqref{coef}.
\end{proof} 

\section{Proof of the main theorem}\label{sec:preuve_th_princ} Fix an integer $n_0$ and let $x_1$ and $x_2$ be two elements of $G_{n_0}.$ Let $k_1=\mathrm{k}(\bar{x_1}^{n_0})$ and $k_2=\mathrm{k}(\bar{x_2}^{n_0}).$ The product $\mathfrak{a}_{(x_1;k_1)}(n_0)\cdot \mathfrak{a}_{(x_2;k_2)}(n_0)$ is given by Equation \eqref{E}. If we apply $\psi_{n_0}$ to this product then, due to Proposition \ref{prop:compa_psi_n}, we get the following equation:
$$\psi_{n_0}(\mathfrak{a}_{(x_1;k_1)}(n_0))\cdot \psi_{n_0}(\mathfrak{a}_{(x_2;k_2)}(n_0))=\sum_{\max(k_1,k_2)\leq k\leq \min(k_1+k_2,n_0) \atop{x\in G_k}}c_{(x_1;k_1),(x_2;k_2)}^{(x;k)}(n_0)\psi_{n_0}(\mathfrak{a}_{(x;k)}(n_0)).$$
We have,
$$\psi_{n_0}(\mathfrak{a}_{(x;k)}(n_0))=\sum_{C,C'}\frac{1}{|C||C'|}\sum_{c,c'}cxc'=\frac{1}{|K_{n_0-k}|^2}\sum_{h\in K_{n_0},h'\in K_{n_0}}hxh'=\frac{|K_{n_0}|^2}{|K_{n_0}xK_{n_0}||K_{n_0-k}|^2}\overline{\bf x}^{n_0}.$$
Thus, we have :
\begin{eqnarray*}
\frac{|K_{n_0}|^2}{|K_{n_0}x_1K_{n_0}||K_{n_0-k_1}|^2}\overline{\bf x_1}^{n_0}&\cdot &\frac{|K_{n_0}|^2}{|K_{n_0}x_2K_{n_0}||K_{n_0-k_2}|^2}\overline{\bf x_2}^{n_0}=\\
&&\sum_{\max(k_1,k_2)\leq k\leq \min(k_1+k_2,n_0) \atop{x\in G_k}}c_{(x_1;k_1),(x_2;k_2)}^{(x;k)}(n_0)\frac{|K_{n_0}|^2}{|K_{n_0}xK_{n_0}||K_{n_0-k}|^2}\overline{\bf x}^{n_0},
\end{eqnarray*}
which gives us :
$$\overline{\bf x_1}^{n_0}\cdot \overline{\bf x_2}^{n_0}=\sum_{\max(k_1,k_2)\leq k\leq \min(k_1+k_2,n_0) \atop{x\in G_k}}c_{(x_1;k_1),(x_2;k_2)}^{(x;k)}(n_0)\frac{|K_{n_0}x_1K_{n_0}||K_{n_0-k_1}|^2|K_{n_0}x_2K_{n_0}||K_{n_0-k_2}|^2}{|K_{n_0}|^2|K_{n_0}xK_{n_0}||K_{n_0-k}|^2}\overline{\bf x}^{n_0}.$$
By the formula for the structure coefficients $c_{(x_1;k_1),(x_2;k_2)}^{(x;k)}(n_0)$ given in Proposition \ref{F}, we get :
\begin{eqnarray*}
&&\overline{\bf x_1}^{n_0}\cdot \overline{\bf x_2}^{n_0}=\\
&&\sum_{\max(k_1,k_2)\leq k\leq \min(k_1+k_2,n_0) \atop{x\in G_k,\text{ $x_1^{-1}xx_2^{-1}$ is in $K_k$ and $(k_1,k_2)$-minimal}}}\frac{|K_{n_0}x_1K_{n_0}||K_{n_0-k_1}||K_{n_0}x_2K_{n_0}||K_{n_0-k_2}|}{|K_{n_0}||K_{n_0}xK_{n_0}||K_{n_0-k}||Cl_{k_1,k_2}(x_1^{-1}xx_2^{-1})\cap K_{m_{k_1,k_2}(x_1^{-1}xx_2^{-1})}|}\overline{\bf x}^{n_0}.
\end{eqnarray*}

To obtain the expression of the structure coefficients given in Theorem \ref{main_th}, we should remark that while fixing $\bar{x_3}^{n_0},$ to obtain its coefficient we must sum over all the $x$'s such that $\bar{x}^{n_0}=\bar{x_3}^{n_0}$ which implies that $\mathrm{k}(\bar{x}^{n_0})=\mathrm{k}(\bar{x_3}^{n_0}).$ Thus the $x$'s which appear in the sum must be in $G_k$ where $k\geq \mathrm{k}(\bar{x_3}^{n_0})$. 

\begin{rem}
In \cite{toutejc}, to prove the main theorem about the polynomiality property of the structure coefficients of the Hecke algebra of the pair \hecke, we build a universal algebra  $\mathcal{A'}_\infty$ which projects on the Hecke algebra of the pair $(\mathcal{S}_{2n},\mathcal{B}_n)$ for each $n.$ That algebra is isomorphic, as it is shown in the same paper, to the algebra of $2$-shifted symmetric functions. In \cite{Ivanov1999} also, Ivanov and Kerov build a similar universal algebra to prove the polynomiality property of the structure coefficients of the center of the symmetric group algebra. Méliot also build a universal algebra, see \cite{meliot2013partial}, to prove the polynomiality property of the structure coefficients of the center of the group of invertible matrices with coefficients in a finite field algebra.

It is also possible to build a universal "non-associative algebra" in our general framework due to Proposition \ref{F}. What is remarkable in our proof is that we do not need to build such an algebra to give the form of the structure coefficients. In fact, we could have done the same (obtain the polynomiality property without the construction of a universal algebra) in our paper \cite{toutejc}. Using the formula of $|H_{\lambda\delta}^{\rho}(n)|$ on page 23 in that paper and the formula on page 21 linking $c_{\lambda,\delta}^\rho(n)$ with the cardinal of $H_{\lambda\delta}^{\rho}(n)$, we obtain directly a result about the dependence on $n$ of $c_{\lambda,\delta}^\rho(n).$ By applying the homomorphism given in Section 3.6, we get our result about the polynomiality of the structure coefficients of the Hecke algebra of the pair $(\mathcal{S}_{2n},\mathcal{B}_n)$ without the construction of a universal algebra. We used a similar idea in this paper to give a polynomiality property for the structure coefficients without building a universal algebra.
\end{rem}

\section{Applications and results of polynomiality}\label{sec:appl_double_classe}
We recall here some definitions concerning partitions since they will be used to index the bases of algebras considered throughout this section.

A \textit{partition} $\lambda=(\lambda_1,\lambda_2,\cdots,\lambda_r)$ is a decreasing sequence of positive integers. The $\lambda_i$ are called the \textit{parts} of the partition $\lambda.$ The \textit{size} of $\lambda,$ which is denoted $|\lambda|,$ is the sum of all the $\lambda_i.$ The \textit{length} of $\lambda,$ denoted $l(\lambda),$ is the number $r$ of its parts. We say that $\lambda$ is a partition of $n,$ and we write $\lambda\vdash n,$ if $|\lambda|=n.$ If $m_i(\lambda)$ is the number of parts in $\lambda$ equals to $i,$ then $\lambda$ can be written in an exponential way as follows:
$$\lambda=(1^{m_1(\lambda)},2^{m_2(\lambda)},\cdots).$$
We will denote by $\mathcal{P}_n$ the set of partitions of $n.$

For a partition $\lambda,$ we define the number $z_\lambda$ as follows:
$$z_\lambda=\prod_{i\geq 1}i^{m_i(\lambda)}m_i(\lambda)!.$$

The union of two partitions $\lambda=(1^{m_1(\lambda)},2^{m_2(\lambda)},\cdots)$ and $\delta=(1^{m_1(\delta)},2^{m_2(\delta)},\cdots)$ is the partition obtained by joining the parts of $\lambda$ and $\delta,$ explicitly:
$$\lambda\cup \delta:=(1^{m_1(\lambda)+m_1(\delta)},2^{m_2(\lambda)+m_2(\delta)},\cdots).$$

A \textit{proper partition} is a partition without parts equal to $1.$ The set of proper partitions of size $n$ will be denoted by $\mathcal{PP}_n.$ The set $\mathcal{P}_n$ of partitions of $n$ is in bijection with the set $\mathcal{PP}_{\leq n}$ defined by:
$$\mathcal{PP}_{\leq n}:=\bigsqcup_{0\leq r \leq n}\mathcal{PP}_r.$$
This bijection is:
$$\begin{array}{ccccc}
&  & \mathcal{P}_n & \longrightarrow & \mathcal{PP}_{\leq n} \\
& & \lambda & \longmapsto & \overline{\lambda}:=(1^0,2^{m_2(\lambda)},\cdots)\\
\end{array}$$
and its inverse is:
$$\begin{array}{ccccc}
&  & \mathcal{PP}_{\leq n} & \longrightarrow & \mathcal{P}_{n} \\
& & \lambda & \longmapsto & \underline{\lambda}_n:=(1^{n-|\lambda|},2^{m_2(\lambda)},\cdots)\\
\end{array}.$$

\subsection{Hypotheses H.1 to H.5 of our general framework in the case of the symmetric group}\label{sec:hyp_grp_sym}

We show in this subsection that the symmetric group $\mathcal{S}_n$ satisfies all necessary conditions on the sequence of subgroups $K_n$ (that means hypotheses H.1 to H.5; we will not check H.0 here because it also depends on $G_n$).

Let $1\leq k\leq n$, we define $\mathcal{S}_n^k$ to be the symmetric group which acts on the last $n-k$ elements of the set $[n]$. Explicitly,
$$\mathcal{S}_n^k:=\lbrace x\in \mathcal{S}_n \text{ such that }x(1)=1,\, x(2)=2, \cdots , x(k)=k\rbrace.$$
Clearly $\mathcal{S}_n^k$ is isomorphic to $\mathcal{S}_{n-k}$ for each $1\leq k\leq n$ and thus we have H.1. For each $x\in \mathcal{S}_k$ and for any $y\in \mathcal{S}_{n-k}$, the composition of $x$ and $y$ commutes since $x$ and $y$ act on disjoint sets. That means that we also have H.2. If $1\leq k\leq n,$ then $\mathcal{S}_{n+1}^{k}\cap \mathcal{S}_n$ is the set of permutations of $n+1$ which fix $1,2,\cdots,k$ and $n+1$, therefore $\mathcal{S}_{n+1}^{k}\cap \mathcal{S}_n=\mathcal{S}_{n}^{k}$ and we have H.3. The other three hypotheses needed are proven in the lemmas below.


\begin{lem}\label{lem:k<k_1+k_2_group_sym} (H.4 for $\mathcal{S}_n$)
Let $z\in \mathcal{S}_n$, then we have :
$$\mathrm{k}(\mathcal{S}_n^{k_1} z\mathcal{S}_n^{k_2})\leq |\lbrace 1,\cdots,k_2,z(1),\cdots,z(k_1)\rbrace|\leq k_1+k_2.$$
\end{lem}
\begin{proof}
It is convenient for us to use the two line notation of permutations in this proof. 
The set $\mathcal{S}_n^{k_1}z$ contains permutations of the following form :
$$\begin{matrix}
1 & 2 & \cdots & k_1 & k_1+1 & \cdots & k_1+k_2  & \cdots & n \\
z(1) & z(2) & \cdots & z(k_1) & * & \cdots & * & \cdots & *
\end{matrix}.$$
The stars are used to say that the images are not fixed. Explicitly, we have : 
$$\mathcal{S}_n^{k_1}z=\lbrace x\in S_n \text{ such that } x(i)=z(i) \text{ for } i=1,\cdots, k_1\rbrace$$
and
$$z\mathcal{S}_n^{k_2}=\lbrace x\in S_n \text{ such that } x(z^{-1}(i))=z(i) \text{ for } i=1,\cdots, k_2\rbrace.$$ 
Then we can explicitly write :
$$\mathcal{S}_n^{k_1}z\mathcal{S}_n^{k_2}=\bigcup_{x\in \mathcal{S}_n^{k_1}z}\lbrace y\in S_n \text{ such that } y(x^{-1}(i))=x(i) \text{ for } i=1,\cdots, k_2\rbrace.$$
Let us denote by $r$ the size of the set $\lbrace 1,\cdots,k_2\rbrace \cap \lbrace z(1),\cdots,z(k_1)\rbrace$ and suppose that $\lbrace h_1,\cdots, h_{k_2-r}\rbrace =\lbrace1,\cdots,k_2\rbrace\setminus \lbrace z(1),\cdots,z(k_1)\rbrace.$ We can find a permutation of the following form
$$\begin{matrix}
1 & 2 & \cdots & k_1 & k_1+1 & \cdots & k_1+k_2-r & k_1+k_2-r+1 & \cdots & n \\
z(1) & z(2) & \cdots & z(k_1) & h_1 & \cdots & h_{k_2-r} & * & \cdots & *
\end{matrix}$$
in $\mathcal{S}_n^{k_1}z.$ Since the multiplication by an element of $\mathcal{S}_n^{k_2}$ to the right permutes the elements greater than $k_2$ in the second line defining this permutation, the set $\mathcal{S}_n^{k_1}z\mathcal{S}_n^{k_2}$ contains thus a permutation of the following form
$$\begin{matrix}
1 & 2 & \cdots & k_1 & k_1+1 & \cdots & k_1+k_2-r & k_1+k_2-r+1 & \cdots & n \\
* & ** & \cdots & * & h_1 & \cdots & h_{k_2-r} & k_1+k_2-r+1 & \cdots & n
\end{matrix}$$
This permutation is also in $\mathcal{S}_{k_1+k_2-r}.$ We put $**$ to say that there are $r$ fixed elements (elements among $z(1),z(2), \cdots, z(k_1)$ smaller than $k_2+1$ can change position after multiplication on right by $\mathcal{S}_n^{k_2}$) in the $k_1$ first images but we are not interested in their positions. The fact that $k_1+k_2-r=|\lbrace 1,\cdots,k_2,z(1),\cdots,z(k_1)\rbrace|$ ends the proof.
\end{proof}
\begin{lem}\label{lem:hyp_6_grp_sym}
Let $z$ be an element of $\mathcal{S}_n$, then we have :
$$z\mathcal{S}_n^{k_1}z^{-1}\cap \mathcal{S}_n^{k_2}\simeq \mathcal{S}_n^{r(z)},$$
where \begin{eqnarray*}
r(z)&=&|\lbrace z^{-1}(1),z^{-1}(2),\cdots , z^{-1}(k_1),1,\cdots ,k_2\rbrace| \\
&=&k_1+k_2-|\lbrace z^{-1}(1),z^{-1}(2),\cdots , z^{-1}(k_1)\rbrace\cap\lbrace 1,\cdots ,k_2\rbrace|.
\end{eqnarray*}
If $z$ is $(k_1,k_2)$-minimal, then $r(z)=\mathrm{k}(\mathcal{S}_n^{k_1}z\mathcal{S}_n^{k_2})$ which gives us hypothesis H.5 in the case of $\mathcal{S}_n.$
\end{lem}
\begin{proof}
Let $a=zbz^{-1}$ be an element of $\mathcal{S}_n$ which fixes the $k_2$ first elements while $b$ fixes the $k_1$ first elements. Then $a$ also fixes the elements $z^{-1}(1),\cdots, z^{-1}(k_1)$ which proves that $z\mathcal{S}_n^{k_1}z^{-1}\cap \mathcal{S}_n^{k_2}\subseteq \mathcal{S}_n^{r(z)}.$ In the opposite direction, if $x$ is a permutation of $n$ which fixes the elements of the set $\lbrace z^{-1}(1),z^{-1}(2),\cdots , z^{-1}(k_1),1,\cdots ,k_2\rbrace$ then $x$ is in $\mathcal{S}_n^{k_2}$ and in addition $z^{-1}xz$ is in $\mathcal{S}_n^{k_1}$ which implies that $x=zz^{-1}xzz^{-1}$ is in $z\mathcal{S}_n^{k_1}z^{-1}.$
\end{proof}

\subsection{Hypotheses H.1 to H.5 of our general framework in the case of the hyperoctahedral group}\label{sec:cond_group_hyp}

Here we show that the hyperoctahedral group $\mathcal{B}_n$ fulfils the hypotheses on the sequence of the sub-groups $K_n$ (that means hypotheses H.1 to H.5; we are not interested in H.0 since it also depends on $G_n$).
Let $1\leq k\leq n$, the set $\mathcal{B}_n^k$ represents the hyperoctahedral sub-group of $\mathcal{S}_{2n}$ which acts on the $2n-2k$ last elements of the set $[2n].$ Explicitly,
$$\mathcal{B}_n^k:=\lbrace x\in \mathcal{S}_{2n} \text{ such that }x(1)=1,\, x(2)=2, \cdots , x(2k)=2k\rbrace.$$
It is evident that $\mathcal{B}_n^k$ is isomorphic to $\mathcal{B}_{n-k}$ for each $1\leq k\leq n$ and thus H.1 is satisfied. For each $x\in \mathcal{B}_k$ and for any $y\in \mathcal{B}_{n-k}$, the composition of $x$ and $y$ commutes because the two permutations $x$ and $y$ act on disjoint sets, in other words H.2 is also satisfied. If $1\leq k\leq n,$ then $\mathcal{B}_{n+1}^{k}\cap \mathcal{B}_n$ is the set of permutations of $2n+2$ which fix $1,2,\cdots,2k, 2n+1$ and $2n+2$, thus $\mathcal{B}_{n+1}^{k}\cap \mathcal{B}_n=\mathcal{B}_{n}^{k}$ and H.3 is satisfied. The other necessary hypotheses are proven in the following two lemmas.

\begin{lem} (H.4 for $\mathcal{B}_n$)
Let $z\in \mathcal{B}_n$, then we have :
$$\mathrm{k}(\mathcal{B}_n^{k_1} z\mathcal{B}_n^{k_2})\leq \frac{|\lbrace 1,\cdots,2k_2,z(1),\cdots,z(2k_1)\rbrace|}{2}\leq k_1+k_2.$$
\end{lem}
\begin{proof}
The proof of this lemma is similar to that of Lemma \ref{lem:k<k_1+k_2_group_sym}.
\end{proof}
\begin{lem} (H.5 for $\mathcal{B}_n$)
Let $z$ be an element of $\mathcal{B}_n$, then we have :
$$z\mathcal{B}_n^{k_1}z^{-1}\cap \mathcal{B}_n^{k_2}\simeq \mathcal{B}_n^{r(z)},$$
where \begin{eqnarray*}
r(z)&=&|\lbrace z^{-1}(1),z^{-1}(2),\cdots , z^{-1}(2k_1),1,\cdots ,2k_2\rbrace| \\
&=&2k_1+2k_2-|\lbrace z^{-1}(1),z^{-1}(2),\cdots , z^{-1}(2k_1)\rbrace\cap\lbrace 1,\cdots ,2k_2\rbrace|.
\end{eqnarray*}
If $z$ is $(k_1,k_2)$-minimal, then $r(z)=\mathrm{k}(\mathcal{B}_n^{k_1}z\mathcal{B}_n^{k_2}),$ thus we have H.5 for $\mathcal{B}_n.$
\end{lem}
\begin{proof}
The proof of this lemma is similar to that of Lemma \ref{lem:hyp_6_grp_sym}. 
\end{proof}

\subsection{The Hecke algebra of the pair $(\mathcal{S}_{2n},\mathcal{B}_n)$}\label{sec:alg_de_hecke}

As we have seen, the hyperoctahedral group $\mathcal{B}_n$ satisfies hypotheses H.1 to H.5 required in Section \ref{sec:hypo_defi}. To apply Theorem \ref{mini_th}, we also need to verify H.0 in the case of the sequence $(\mathcal{S}_{2n},\mathcal{B}_n)_n.$ In other words, we should prove that for any permutation $x$ of $\mathcal{S}_{2n}$ we have :
$$\mathcal{B}_{n+1}x\mathcal{B}_{n+1}\cap \mathcal{S}_{2n}=\mathcal{B}_{n}x\mathcal{B}_{n}.$$
A proof for hypothesis H.0 for the sequence $(\mathcal{S}_{2n},\mathcal{B}_n)_n$ can be given using the combinatorial description of the $\mathcal{B}_n$-double-classes. In fact, if $x\in \mathcal{S}_{2n}$ then the double-class $\mathcal{B}_nx\mathcal{B}_n$ is the set of permutations of $2n$ with coset-type\footnote{We refer to \cite[page 401]{McDo} for a definition of the coset-type of a permutation of $2n$} equals to $\ct(x).$ Likewise, the double-class $\mathcal{B}_{n+1}x\mathcal{B}_{n+1}$ ($x$ is now seen as a permutation of $2n+2$) is the set of permutations of $2n+2$ with coset-type equals to $\ct(x)\cup (1)$ which corresponds to $\mathcal{B}_nx\mathcal{B}_n$ when we take the intersection with $\mathcal{S}_{2n}.$

Let $\lambda$ be an element of $\mathcal{PP}_{\leq n}$, the double-class of $\mathcal{B}_n$ in $\mathcal{S}_{2n}$ associated to $\lambda$ is, according to \cite{toutejc}, as follows :
$${K}_{\underline{\lambda}_n}=\lbrace \omega\in \mathcal{S}_{2n}\text{ such that } \ct(\omega)=\lambda\cup (1^{n-|\lambda|})\rbrace.$$
The size of ${K}_{\underline{\lambda}_n}$ is given by the following formula :
\begin{equation}
|{K}_{\underline{\lambda}_n}|=\frac{(2^nn!)^2}{z_{2\lambda}2^{n-|\lambda|}(n-|\lambda|)!}.
\end{equation} 

We fix three proper partitions $\lambda$, $\delta$ and $\rho.$ Let $n$ be an integer sufficiently big. By using Theorem \ref{mini_th}, the coefficient ${\bf K}_{\underline{\rho}_n}$ in the product ${\bf K}_{\underline{\lambda}_n}{\bf K}_{\underline{\delta}_n}$ is of the following form :
\begin{equation}
\frac{\frac{(2^nn!)^2}{z_{2\lambda}2^{n-|\lambda|}(n-|\lambda|)!}\frac{(2^nn!)^2}{z_{2\delta}2^{n-|\delta|}(n-|\delta|)!}2^{n-|\lambda|}(n-|\lambda|)!2^{n-|\delta|}(n-|\delta|)!}{2^nn!\frac{(2^nn!)^2}{z_{2\rho}2^{n-|\rho|}(n-|\rho|)!}}\sum_{|\rho|\leq k\leq |\lambda|+|\delta|}\frac{a_{\lambda\delta}^{\rho}(k)}{2^{n-k}({n-k)!}}
\end{equation}
which is, after simplification, equal to :
\begin{equation}
2^nn!\frac{z_{2\rho}}{z_{2\lambda}z_{2\delta}}\sum_{|\rho|\leq k\leq |\lambda|+|\delta|}a_{\lambda\delta}^{\rho}(k)2^{k-|\rho|}\frac{(n-|\rho|)!}{(n-k)!}.
\end{equation}
This gives us the following corollary.

\begin{cor}
Let $\lambda$ and $\delta$ be two proper partitions. Let $n$ be an integer sufficiently big and consider the following equation :
$${\bf K}_{\underline{\lambda}_n}{\bf K}_{\underline{\delta}_n}=\sum_{\rho \text{ proper partition}} {c'}_{\lambda\delta}^{\rho}(n) {\bf K}_{\underline{\rho}_n}.$$
The coefficients $\frac{{c'}_{\lambda\delta}^{\rho}(n)}{2^nn!}$ are polynomials in $n$ with rational coefficients.
\end{cor}

This corollary is the main result of \cite{toutejc} on the polynomiality of the structure coefficients of the algebra \Hecke. 

Otherwise, Theorem \ref{main_th} can be used to find the exact values of the structure coefficients of the Hecke algebra of the pair $(\mathcal{S}_{2n},\mathcal{B}_n).$ However, this can be done only in particular cases and may be very complicated as it is shown in the following example.

\begin{ex}

Let $k_1=k_2=2,$ $x_1=(1~~2~~4~~3)$ and $x_2=(1~~4~~2)(3).$ Let $n$ be an integer sufficiently big, then $$\sum_{y_1\sim x_1}y_1=\sum_{y_2\sim x_2}y_2={\bf K}_{\underline{(2)}_n}.$$
Since the $k$'s in the sum index of Theorem \ref{main_th} must be less than $k_1+k_2$ and greater than $k_1$ and $k_2$, $k$ is thus either $2$, $3$ or $4.$\\
For $k=2$, all the permutations of $\mathcal{B}_2$ are $(2,2)$-minimals.\\
For $k=3$, the $(2,2)$-minimal permutations are those which are in $\mathcal{B}_3$ and send the set $\lbrace 5,6\rbrace$ to $\lbrace 1,2\rbrace$ or $\lbrace 3,4\rbrace.$\\
For $k=4$, the $(2,2)$-minimal permutations are those which belong to $\mathcal{B}_4$ and send the set $\lbrace 1,2,3,4\rbrace$ to $\lbrace 5,6,7,8\rbrace.$\\
Therefore, for $k=2$, the permutations $x\in \mathcal{S}_4$ such that $x_1^{-1}xx_2^{-1}$ is $(2,2)$-minimal are the permutations of the set $x_1\mathcal{B}_2x_2,$ which are, ${\color{green}(1~~2)(3~~4)}$, ${\color{blue}(1~~2~~4)(3)}$, ${\color{blue}(1)(2~~3~~4)}$, ${\color{green}(1~~4~~2~~3)}$, ${\color{blue}(1~~3~~2)(4)}$, ${\color{green}(1~~3~~2~~4)}$, ${\color{green}(1)(2)(3~~4)}$, ${\color{blue}(1~~4~~3)(2)}$.\\
Likewise, for $k=3$, the permutations $x\in \mathcal{S}_4$ such that $x_1^{-1}xx_2^{-1}$ is $(2,2)$-minimal are :\\
$\begin{pmatrix}
2&6&3&5&4&1\\
3&6&2&5&4&1\\
2&5&3&6&4&1\\
3&5&2&6&4&1\\
5&2&6&3&4&1\\
6&2&5&3&4&1\\
5&3&6&2&4&1\\
6&3&5&2&4&1
\end{pmatrix}$~~$\begin{pmatrix}
2&6&3&5&1&4\\
3&6&2&5&1&4\\
2&5&3&6&1&4\\
3&5&2&6&1&4\\
5&2&6&3&1&4\\
6&2&5&3&1&4\\
5&3&6&2&1&4\\
6&3&5&2&1&4
\end{pmatrix}$~~$\begin{pmatrix}
1&6&4&5&2&3\\
1&5&4&6&2&3\\
4&6&1&5&2&3\\
4&5&1&6&2&3\\
5&1&6&4&2&3\\
5&4&6&1&2&3\\
6&1&5&4&2&3\\
6&4&5&1&2&3
\end{pmatrix}$~~$\begin{pmatrix}
1&6&4&5&3&2\\
1&5&4&6&3&2\\
4&6&1&5&3&2\\
4&5&1&6&3&2\\
5&1&6&4&3&2\\
5&4&6&1&3&2\\
6&1&5&4&3&2\\
6&4&5&1&3&2
\end{pmatrix}.$\\
In every matrix above, each line defines a permutation.\\
For $k=4$, the permutations $x\in \mathcal{S}_4$ such that $x_1^{-1}xx_2^{-1}$ is $(2,2)$-minimal are those of coset-type $(2,2)$ such that the image of $\lbrace 1,2,3,4\rbrace$ is $\lbrace 5,6,7,8\rbrace.$\\
The permutations written in green have $(1^2)$ as coset-type. These permutations give the coefficient of ${\bf K}_{\underline{\emptyset}_n}$. For each one of them, $|Cl_{k_1,k_2}(x_1^{-1}xx_2^{-1})\cap K_{m_{k_1,k_2}(x_1^{-1}xx_2^{-1})}|=1.$ Thus, the coefficient of ${\bf K}_{\underline{\emptyset}_n}$ is, by using Theorem \ref{main_th}, as follows :
$$4\frac{(2^nn!n(n-1))^2(2^{n-2}(n-2)!)^2}{2^nn!2^nn!2^{n-2}(n-2)!}=2^nn!n(n-1).$$
The permutations written in blue have $(2)$ as coset-type. These permutations give the coefficient of ${\bf K}_{\underline{(2)}_n}.$ To each one of them  $|Cl_{k_1,k_2}(x_1^{-1}xx_2^{-1})\cap K_{m_{k_1,k_2}(x_1^{-1}xx_2^{-1})}|=1$. Thus, by Theorem \ref{main_th} also, the coefficient of ${\bf K}_{\underline{(2)}_n}$ is :
$$4\frac{(2^nn!n(n-1))^2(2^{n-2}(n-2)!)^2}{2^nn!2^nn!n(n-1)2^{n-2}(n-2)!}=2^nn!.$$
All permutations written in matrix form above have $(3)$ as coset-type. These permutations give the coefficient of ${\bf K}_{\underline{(3)}_n}$. To each one of them $|Cl_{k_1,k_2}(x_1^{-1}xx_2^{-1})\cap K_{m_{k_1,k_2}(x_1^{-1}xx_2^{-1})}|=4.$ Thus, the coefficient of ${\bf K}_{\underline{(3)}_n}$ is :
$$8.4\frac{(2^nn!n(n-1))^2(2^{n-2}(n-2)!)^2}{2^nn!\frac{4}{3}2^nn!n(n-1)(n-2)2^{n-3}(n-3)!4}=3\cdot2^nn!.$$
The permutations $x\in \mathcal{S}_4$ of coset-type $(2,2)$ such that the image of $\lbrace 1,2,3,4\rbrace$ is $\lbrace 5,6,7,8\rbrace$ give us the coefficient ${\bf K}_{\underline{(2^2)}_n}.$ To each one of them, we have $|Cl_{k_1,k_2}(x_1^{-1}xx_2^{-1})\cap K_{m_{k_1,k_2}(x_1^{-1}xx_2^{-1})}|=64.$ Thus, the coefficient of ${\bf K}_{\underline{(2,2)}_n}$ is :
$$64\frac{(2^nn!n(n-1))^2(2^{n-2}(n-2)!)^2}{2^nn!2^{n-1}n!n(n-1)(n-2)(n-3)2^{n-4}(n-4)!64}=2\cdot2^nn!.$$

Thus we can obtain the complete formula of the product ${\bf K}_{\underline{(2)}_n}\cdot {\bf K}_{\underline{(2)}_n}$ for any $n\geq 4$,
$${\bf K}_{\underline{(2)}_n}\cdot {\bf K}_{\underline{(2)}_n}=2^nn!n(n-1){\bf K}_{\underline{\emptyset}_n}+2^{n}n!{\bf K}_{\underline{(2)}_n}+2^{n}n!3{\bf K}_{\underline{(3)}_n}+2^{n}n!2{\bf K}_{\underline{(2^2)}_n}.$$

\end{ex}

The exact values of the structure coefficients in the product ${\bf K}_{\underline{(2)}_n}\cdot {\bf K}_{\underline{(2)}_n}$ can be found in Example 4.1, page 28, in \cite{toutejc}.

\subsection{The double-class algebra of $\diag(\mathcal{S}_{n-1})$ in $\mathcal{S}_n\times \mathcal{S}_{n-1}^{opp}$}\label{alg_double_classe_diag(S_n-1)}

In this section, we consider $\mathcal{S}_{n-1}$ to be the sub-group $\mathcal{S}_n^1$ of $\mathcal{S}_n.$ That means that $\mathcal{S}_{n-1}$ is the set of permutations of $n$ which fix $1.$ Note that to prove hypotheses H.1 to H.5 for the symmetric group in Section \ref{sec:hyp_grp_sym}, we saw the group $\mathcal{S}_{n-1}$ as the sub-group of $\mathcal{S}_{n}$ of permutations which fix $n.$ By considering $\mathcal{S}_{n-1}$ as the sub-group $\mathcal{S}_n^1$ of $\mathcal{S}_n,$ hypotheses H.1 to H.5 remain valid and their proofs are the same 'up to isomorphisms' as in Section \ref{sec:hyp_grp_sym}.

Hypothesis H.0 is also satisfied in the case of the sequence $(\mathcal{S}_n\times \mathcal{S}_{n-1}^{opp},\diag(\mathcal{S}_{n-1})).$ It is proved by showing that if $(x,y)\in \mathcal{S}_n\times \mathcal{S}_{n-1}^{opp}$ and if $a,b\in \mathcal{S}_n$ such that $(axb,b^{-1}ya^{-1})\in \mathcal{S}_n\times \mathcal{S}_{n-1}^{opp}$ then there exists $a',b'\in \mathcal{S}_{n-1}$ such that $(axb,b^{-1}ya^{-1})=(a'xb',{b'}^{-1}y{a'}^{-1}).$ Instead of giving a direct proof, we will later show H.0 using the combinatorial description of the $\diag(\mathcal{S}_{n-1})$-double-classes. We did the same in Section \ref{sec:alg_de_hecke} for the pair $(\mathcal{S}_{2n},\mathcal{B}_n)$.

The double-class algebra of $\diag(\mathcal{S}_{n-1})$ in $\mathcal{S}_n\times \mathcal{S}_{n-1}^{opp}$ was studied by Brender in $1976,$ see \cite{brender1976spherical}. In $2007,$ Strahov proved, see \cite[Proposition 2.2.1]{strahov2007generalized}, that the pair $(\mathcal{S}_n\times \mathcal{S}_{n-1}^{opp},\diag(\mathcal{S}_{n-1}))$ is a Gelfand pair -- that means that the double-class algebra of $\diag(\mathcal{S}_{n-1})$ in $\mathcal{S}_n\times \mathcal{S}_{n-1}^{opp}$ is commutative -- then he studied the zonal spherical functions associated to this pair. In this section, we are interested in the structure coefficients of this algebra and we will give at the end of this subsection a polynomiality property.

Two permutations $x$ and $y$ of $\mathcal{S}_n$ are conjugated with respect to $\mathcal{S}_{n-1}$ if $x=zyz^{-1}$ for a certain element $z\in S_{n-1}.$ Let $x$ be a permutation of $\mathcal{S}_n$ and let $k$ be a permutation of $\mathcal{S}_{n-1},$ we know that both $x$ and $kxk^{-1}$ have the same cycle-type\footnote{The cycle-type of a permutation of $n$ is the partition of $n$ obtained using the lengths of the cycles that appear in its decomposition into disjoint cycles.}, but in addition if $c=(1,a_2, \cdots ,a_{l(c)})$ is the cycle of $x$ which contains $1$ then the cycle $(1,z^{-1}(a_2) \cdots ,z^{-1}(a_{l(c)}))$ of $zxz^{-1}$ contains $1$ and has the same length as $c.$ On the other hand, if two permutations of $n$ have the same cycle-type and if the cycles, containing $1$ in their cycle decompositions, have the same length, then it is easy to see that these two permutations are conjugated with respect to $\mathcal{S}_{n-1}.$

The conjugacy classes with respect to $\mathcal{S}_{n-1}$ are indexed by pairs $(i,\lambda)$ where $i$ is an integer between $1$ and $n$ and $\lambda$ is a partition of $n-i.$ The conjugacy class with respect to $\mathcal{S}_{n-1}$ associated to the pair $(i,\lambda)$ is as follows :
$$C_{(i,\lambda)}=\lbrace x\in \mathcal{S}_n \text{ such that $1$ is in a cycle $c$ of length $i$ and } cycle-type(x\setminus c)=\lambda\rbrace,$$
where $x\setminus c$ is the permutation obtained from $x$ by removing the cycle $c.$
According to \cite[page 118]{strahov2007generalized}, the size of such a conjugacy class is :
$$|C_{(i,\lambda)}|=\frac{(n-1)!}{z_\lambda}.$$
A detailed study of the conjugacy classes with respect to $\mathcal{S}_{n-1}$ is given by Jackson et Sloss in \cite{jackson2012character} where the authors use the pairs $(\lambda,i),$ where $\lambda$ is a partition of $n$ containing necessarily a part $i$ (the set of these partitions is in bijection with the set of partitions of $n-i$), to index them.

Let $(a,b)$ be an element of $\mathcal{S}_n\times \mathcal{S}^{opp}_{n-1}$ and let $x$ and $y$ be two elements of $\mathcal{S}_{n-1}$, then we have:
$$(x,x^{-1})\cdot (a,b)\cdot (y,y^{-1})=(xay,y^{-1}bx^{-1}).$$
Therefore, two elements $(a,b)$ and $(c,d)$ of $\mathcal{S}_n\times \mathcal{S}^{opp}_{n-1}$ are in the same $diag(\mathcal{S}_{n-1})$-double-class if and only if $ab$ and $cd$ are conjugated with respect to $\mathcal{S}_{n-1}.$

The set of $\diag(\mathcal{S}_{n-1})$-double-classes is thus also indexed by pairs $(i,\lambda)$ where $i$ is an integer between $1$ and $n$ and $\lambda$ is a partition of $n-i.$ The double-class associated to the pair $(i,\lambda)$ is as follows:
$$DC_{(i,\lambda)}=\lbrace (a,b)\in \mathcal{S}_n\times \mathcal{S}^{opp}_{n-1} \text{ such that } ab\in C_{(i,\lambda)} \rbrace.$$

Now, as we already mentioned, we will show hypothesis H.0 for the sequence $(\mathcal{S}_n\times \mathcal{S}_{n-1}^{opp},\diag(\mathcal{S}_{n-1}))$ using this description of the $\diag(\mathcal{S}_{n-1})$-double-classes. Take an element $(a,b)\in \mathcal{S}_n\times \mathcal{S}^{opp}_{n-1}$ and suppose that $ab\in C_{(i,\lambda)}$ for a certain integer $i$ between $1$ and $n$ and a partition $\lambda$ of $n-i.$ Its $\diag(\mathcal{S}_{n})$-double-class \big($(a,b)$ is now seen as an element of $\mathcal{S}_{n+1}\times \mathcal{S}^{opp}_{n}$\big) is $DC_{(i,\lambda\cup (1))}$ which becomes $DC_{(i,\lambda)}$ when we intersect it with $\mathcal{S}_n\times \mathcal{S}_{n-1}^{opp}.$

For any $a\in C_{(i,\lambda)}$ and any $x\in S_{n-1},$ the $(ax,x^{-1})$ are all different elements in $DC_{(i,\lambda)}.$ Thus, we have:
$$|DC_{(i,\lambda)}|=|S_{n-1}||C_{(i,\lambda)}|=\frac{(n-1)!^2}{z_\lambda}.$$

Let $i$ and $j$ be two integers between $1$ and $n$ and let $\lambda$ and $\delta$ be two partitions of $n-i$ and $n-j.$ The structure coefficients $c_{(i,\lambda)(j,\delta)}^{(r,\rho)}$ of the double-class algebra $\mathbb{C}[\diag(\mathcal{S}_{n-1})\setminus \mathcal{S}_n\times \mathcal{S}^{opp}_{n-1}/ \diag(\mathcal{S}_{n-1})]$ are defined by the following equation:

\begin{equation}\label{eq:strc_coef_doub_class_diag}DC_{(i,\lambda)}DC_{(j,\delta)}=\sum_{1\leq k\leq n \atop{\rho\vdash n-r}}c^{(r,\rho)}_{(i,\lambda)(j,\delta)}DC_{(r,\rho)}.
\end{equation}

\begin{definition}
A pair $(i,\lambda),$ where $i$ is an integer and $\lambda$ is a partition, is said to be {\em proper} if the partition $\lambda$ is proper. If $(i,\lambda)$ is a proper pair, for any integer $n\geq i+|\lambda|,$ we define $\underline{(i,\lambda)}_n$ to be the following pair :
$$\underline{(i,\lambda)}_n=(i,\underline{\lambda}_{(n-i)}).$$
\end{definition} 
For any integer $n\geq i+|\lambda|,$ we have:
$$|DC_{\underline{(i,\lambda)}_n}|=\frac{(n-1)!^2}{z_\lambda(n-i)!}.$$
If $(i,\lambda)$ and $(j,\delta)$ are two proper pairs, then by Equation \eqref{eq:strc_coef_doub_class_diag}, for any integer $n\geq i+|\lambda|, j+|\delta|$ we can write :
\begin{equation*}DC_{\underline{(i,\lambda)}_n}DC_{\underline{(j,\delta)}_n}=\sum_{1\leq r\leq n \atop{\rho\in \mathcal{PP}_{\leq n-r}}}c^{(r,\rho)}_{(i,\lambda)(j,\delta)}(n)DC_{\underline{(r,\rho)}_n}.
\end{equation*}

By using our result for the structure coefficients of double-class algebras given in Theorem \ref{mini_th}, there exists rational numbers $a^{(r,\rho)}_{(i,\lambda)(j,\delta)}(k)$ all independent of $n$ such that :

\begin{eqnarray*}
c^{(r,\rho)}_{(i,\lambda)(j,\delta)}(n)&=&\frac{\frac{(n-1)!^2}{z_\lambda(n-i)!}\frac{(n-1)!^2}{z_\delta(n-j)!}(n-1-i-|\lambda|)!(n-1-j-|\delta|)!}{(n-1)!\frac{(n-1)!^2}{z_\rho(n-r)!}}\\
&&~~~~~~~~~~~~~~~~~~~~~~~~~~\sum_{ r+|\rho|\leq k\leq \min(i+|\lambda|+j+|\delta|,n)}\frac{a^{(r,\rho)}_{(i,\lambda)(j,\delta)}(k)}{(n-1-k)!}\\
&=&\frac{z_\rho(n-1)!}{z_\lambda z_\delta}\frac{1}{(n-i-|\lambda|)\cdots (n-i)\cdot (n-j-|\delta|)\cdots (n-j)} \\
&&~~~~~~~~~~~~~~~~~~~~~~~~~~\sum_{ r+|\rho|\leq k\leq \min(i+|\lambda|+j+|\delta|,n)}a^{(r,\rho)}_{(i,\lambda)(j,\delta)}(k) (n-k)\cdots (n-r).
\end{eqnarray*} 

\begin{cor}
Let $\lambda, \delta$ and $\rho$ be three proper partitions, $i,j$ and $r$ three integers and let $n$ be an integer greater than $i+|\lambda|, j+|\delta|$ and $r+|\rho|,$ then the quotient

\begin{equation*}
(n-i-|\lambda|)\cdots (n-i)\cdot (n-j-|\delta|)\cdots (n-j)\cdot \frac{c^{(r,\rho)}_{(i,\lambda)(j,\delta)}(n)}{(n-1)!}
\end{equation*}
is a polynomial in $n$ with degree less or equal to $i+|\lambda|+j+|\delta|-r+1.$
\end{cor}

The structure coefficients $c^{(r,\rho)}_{(i,\lambda)(j,\delta)}(n)$ have a combinatorial interpretation via special graphs called \textit{dipoles}, see the paper \cite{Jackson20121856}
of Jackson and Sloss for more details about this fact. Note that these two authors also gave in \cite{jackson2012character} a theorem similar to that of Frobenius which writes the structure coefficients of the double-class algebra of $\diag(\mathcal{S}_{n-1})$ in $\mathcal{S}_n\times \mathcal{S}_{n-1}^{opp}$ in terms of generalised characters of the symmetric group.

We recalled in the introduction that the polynomiality property for the structure coefficients was used in the study of the asymptotic behaviour of some combinatorial objects related to the considered algebra. Since the polynomiality property for the structure coefficients of the double-class algebra of $\diag(\mathcal{S}_{n-1})$ in $\mathcal{S}_n\times \mathcal{S}_{n-1}^{opp}$ appears --according to the author's knowledge-- for the first time in this paper, it should be interesting to answer to the following question:

\begin{que}
Does the polynomiality property may be applied to study the asymptotic behaviour of some combinatorial objects related to the double-class algebra of $\diag(\mathcal{S}_{n-1})$ in $\mathcal{S}_n\times \mathcal{S}_{n-1}^{opp}$ ?
\end{que}

\section{A general framework in the case of centers of group algebras}\label{sec:appl_centr}

In \cite[Example 9 page 396]{McDo}, the author showed that the center of a finite group $G$ algebra can be seen as the double-class algebra of $\diag(G)$ in $G\times G^{opp}.$ In \cite{toutAFrobe} as well as in the author Phd thesis \cite[Section 1.3.2]{touPhd14} we give more details about this fact. This allows us to give "center" version of Theorem \ref{main_th} (a theorem about the form of the structure coefficients in case of a sequence of centers of finite groups). 

We consider a sequence $(G_n)_n$ where $G_n$ is a group for each $n.$ Take then the special sequence of $(G_n\times G_n^{opp},\diag(G_n))_n$ pairs. To apply Theorem \ref{main_th}, this sequence must satisfy hypotheses H.0 to H.5 in Section \ref{sec:hypo_defi}. As we have already mentioned, hypotheses H.1 to H.5 only involve the sequence $(\diag(G_n))_n$ (equivalently the sequence $(G_n)_n$). Let us now see what does it mean that $(G_n\times G_n^{opp},\diag(G_n))_n$ satisfies H.0. We show the following lemma :
\begin{lem}
The sequence $(G_n\times G_n^{opp},\diag(G_n))_n$ satisfies hypothesis H.0 if and only if the sequence $(G_n)_n$ satisfies the following hypothesis $\text{H}^{'}.0$ :
\begin{enumerate}
\item[$\text{H}^{'}.0$]\label{hyp'0} $C_g(n+1) \cap G_n=C_g(n)$ for any $g\in G_n,$ where $C_g(n)$ is the conjugacy class of $g$ in $G_n.$
\end{enumerate}
\end{lem} 
\begin{proof}
In fact, if $(G_n\times G_n^{opp},\diag(G_n))_n$ satisfies H.0 and if $y=xgx^{-1}$ is an element of $C_g(n+1)\cap G_n$ with $g\in G_n$ and $x\in G_{n+1}$ then $$(1,y)=(x^{-1},x)(1,g)(x,x^{-1})\in \diag(G_{n+1})(1,g)\diag(G_{n+1})\cap G_n\times G_n^{opp}.$$ But $\diag(G_{n+1})(1,g)\diag(G_{n+1})\cap G_n\times G_n^{opp}$ is $\diag(G_{n})(1,g)\diag(G_{n})$ by H.0. That means that there exists $x'\in G_n$ such that $y=x'g{x'}^{-1}$ and thus $y\in C_g(n).$ Therefore, if $(G_n\times G_n^{opp},\diag(G_n))_n$ satisfies H.0 then $(G_n)_n$ satisfies $\text{H}^{'}.0.$ Reciprocally, if $(G_n)_n$ satisfies $\text{H}^{'}.0$ and if $(x,y)\in \diag(G_{n+1})(g,f)\diag(G_{n+1})$ with $(x,y),(g,f)$ in $G_n\times G_n^{opp}$ then there exists $t\in G_{n+1}$ and $r\in G_{n+1}$ such that :
$$(x,y)=(t,t^{-1})(g,f)(r,r^{-1})=(tgr,r^{-1}ft^{-1}).$$
Thus $xy=tgft^{-1}\in C_{gf}(n+1)\cap G_n$ which is $C_{gf}(n)$ by $\text{H}^{'}.0$ (because $gf\in G_n$), consequently $(x,y)\in \diag(G_{n})(g,f)\diag(G_{n})$ which ends the proof.
\end{proof}

\begin{theoreme}\label{center_main_th}
Let $(G_n)_n$ be a sequence of finite groups satisfying hypothesis $\text{H}^{'}.0$ and the other hypotheses H.1 to H.5 of Section \ref{sec:hypo_defi}. Let $f,$ $h$ and $g$ be three elements of $G_{n_0}$ for a fixed integer and let $k_1=\mathrm{k}(C_{f}(n_0)),$ $k_2=\mathrm{k}(C_{h}(n_0))$ and $k_3=\mathrm{k}(C_{g}(n_0)).$ The structure coefficient ${c}_{f,h}^{g}(n_0)$ of ${\bf C}_{g}(n_0)$ in the product ${\bf C}_{f}(n_0){\bf C}_{h}(n_0)$ is given by the following formula :
\begin{eqnarray*}
{c}_{f,h}^{g}(n_0)&=&\frac{|C_{f}(n_0)||C_{h}(n_0)||G_{n_0-k_1}||G_{n_0-k_2}|}{|G_{n_0}||C_{g}(n_0)|} \nonumber \\
&&\sum_{ \max(k_1,k_2,k_3)\leq k\leq \min(k_1+k_2,n_0), x\in G_k, \atop{\text{ $f^{-1}xh^{-1}\in G_k$ and is $(k_1,k_2)$-minimal} \atop{xhx^{-1}f\in C_{g}(n_0)}}}\frac{1}{|G_{n_0-k}||G_{n_0}^{k_1}f^{-1}xh^{-1}G_{n_0}^{k_2}\cap G_{m_{k_1,k_2}(f^{-1}xh^{-1})}|}.
\end{eqnarray*}
\end{theoreme}
\begin{proof}
The formula for ${c}_{f,h}^{g}(n_0)$ is obtained directly from that in Theorem \ref{main_th}, when applied to the particular sequence $(G_n\times G_n^{opp},\diag(G_n))_n,$ by using Propositions $1.10$ and $1.11$ in \cite{touPhd14}. Since we suppose that the sequence $(G_n)_n$ satisfies hypotheses $\text{H}^{'}.0$ and H.1 to H.5, the sequence $(G_n\times G_n^{opp},\diag(G_n))_n$ satisfies hypotheses H.0 to H.5. To get Theorem \ref{center_main_th}, we apply Theorem \ref{main_th} to the sequence $(G_n\times G_n^{opp},\diag(G_n))_n$ while taking the elements $(f,1),$ $(h,1)$ and $(g,1).$ We obtain :
\begin{tiny}
$$c_{(f,1),(h,1)}^{(g,1)}(n_0)=\frac{|DC_{(f,1)}(n_0)||DC_{(h,1)}(n_0)||\diag(G)_{n_0-k_1}||\diag(G)_{n_0-k_2}|}{|\diag(G)_{n_0}||DC_{(g,1)}(n_0)|} $$
$${\small \sum_{ \max(k_1,k_2,k_3)\leq k\leq \min(k_1+k_2,n_0), (x,y)\in G_k\times G_k, \atop{\text{ $(f^{-1}xh^{-1},y)\in \diag(G_k)$ and is $(k_1,k_2)-minimal$} \atop{DC_{(x,y)}(n_0)=DC_{(g,1)}(n_0)}}}\frac{1}{|\diag(G)_{n_0-k}||\diag(G)_{n_0}^{k_1}(f^{-1}xh^{-1},y)\diag(G)_{n_0}^{k_2}\cap \diag(G)_{m_{k_1,k_2}(f^{-1}xh^{-1},y)}|}.}$$
\end{tiny}
The condition $(f^{-1}xh^{-1},y)\in \diag(G_k)$ is equivalent to $y=hx^{-1}f$ and thus it allows us to take the sum (in the above equation) over the elements of $G_k.$ In addition, the condition $DC_{(x,y)}(n_0)=DC_{(g,1)}(n_0)$ is equivalent to $xhx^{-1}f\in C_g(n_0).$ By \cite[Proposition 1.11]{touPhd14}, we have :
$${c}_{f,h}^{g}(n_0)=\frac{{c}_{(f,1),(h,1)}^{(g,1)}(n_0)}{|G_{n_0}|}.$$ Using \cite[Proposition 1.10]{touPhd14} and after simplification, we get :
\begin{eqnarray*}
{c}_{f,h}^{g}(n_0)&=&\frac{|C_{f}(n_0)||C_{h}(n_0)||G_{n_0-k_1}||G_{n_0-k_2}|}{|G_{n_0}||C_{g}(n_0)|} \nonumber \\
&&\sum_{ \max(k_1,k_2,k_3)\leq k\leq \min(k_1+k_2,n_0), x\in G_k, \atop{\text{ $f^{-1}xh^{-1}\in G_k$ and is $(k_1,k_2)$-minimal} \atop{xhx^{-1}f\in C_{g}(n_0)}}}\frac{1}{|G_{n_0-k}||G_{n_0}^{k_1}f^{-1}xh^{-1}G_{n_0}^{k_2}\cap G_{m_{k_1,k_2}(f^{-1}xh^{-1})}|}.
\end{eqnarray*}
This ends the proof.
\end{proof}

\begin{theoreme}\label{center_mini_th}
Let $(G_n)_n$ be a sequence of finite groups satisfying hypotheses $\text{H}^{'}.0$ and H.1 to H.5 in Section \ref{sec:hypo_defi}. Let $f,$ $h$ and $g$ be three elements of $G_{n_0}$ for a fixed integer $n_0$ and let $k_1=\mathrm{k}(C_f(n_0)),$ $k_2=\mathrm{k}(C_h(n_0))$ and $k_3=\mathrm{k}(C_g(n_0)).$ For any $n\geq n_0,$ the structure coefficient $c_{f,h}^g(n)$ of $\mathbf{C}_g(n)$ in the product $\mathbf{C}_f(n)\mathbf{C}_h(n)$ in the center of the group $G_n$ algebra can be written as follows :

\begin{equation}
c_{f,h}^g(n)=\frac{|C_f(n)||C_h(n)||G_{n-k_1}||G_{n-k_2}|}{|G_n||C_g(n)|}\sum_{ k_3\leq k\leq \max(k_1+k_2,n)}\frac{a_{f,h}^{g}(k)}{|G_{n-k}|},
\end{equation}
where the numbers $a_{f,h}^{g}(k)$ are positive, rational and independent of $n.$
\end{theoreme}
\begin{proof}
This is a direct consequence of Theorem \ref{center_main_th}.
\end{proof}

\subsection{The center of the symmetric group algebra}\label{sec:cent_grp_symetrique}

Recall that we showed in Section \ref{sec:hyp_grp_sym} that the symmetric group satisfies hypotheses H.1 to H.5. To apply our result, we should also verify that the sequence $(Z(\mathbb{C}[\mathcal{S}_n]))_n$ satisfies $\text{H}^{'}.0.$ Let $\omega$ be a permutation of $n,$ the conjugacy class $C_\omega(n)$ of $\omega$ in $\mathcal{S}_n$ is the set of permutations of $n$ with the same cycle-type as $\omega.$ Likewise, by looking at $\omega$ as a permutation of $n+1,$ the conjugacy class $C_\omega(n+1)$ of $\omega$ is the set of permutations of $n+1$ with cycle-type equals cycle-type$(\omega)\cup (1)$ which corresponds to $C_\omega(n)$ when we take the intersection with $\mathcal{S}_n.$ Thus, the sequence $(Z(\mathbb{C}[\mathcal{S}_n]))_n$ also satisfies $\text{H}^{'}.0$ and we can apply Theorem \ref{center_mini_th} in this case.

We recall that the family $({\bf C}_{\underline{\lambda}_n})_{\lambda\in \mathcal{PP}_{\leq n}},$ where
$${C}_{\underline{\lambda}_n}=\lbrace \omega\in \mathcal{S}_n \text{ such that } cycle-type(\omega)=\lambda\cup (1^{n-|\lambda|})\rbrace,$$ 
forms a basis for the center of the symmetric group algebra $Z(\mathbb{C}[\mathcal{S}_n]).$

The size of ${C}_{\underline{\lambda}_n}$ is known to be:
$$|{C}_{\underline{\lambda}_n}|=\frac{n!}{z_\lambda \cdot (n-|\lambda|)!}.$$ 

Let $\lambda$ and $\delta$ be two proper partitions. Let $n$ be an integer sufficiently big, we will apply Theorem \ref{center_mini_th} in the case of a sequence of symmetric groups. For a fixed proper partition $\rho$, the coefficient of ${\bf C}_{\underline{\rho}_n}$ in the expansion of the product ${\bf C}_{\underline{\lambda}_n}{\bf C}_{\underline{\delta}_n}$ is, by Theorem \ref{center_mini_th}, as follows :

\begin{equation}
\frac{\frac{n!}{z_\lambda(n-|\lambda|)!}\frac{n!}{z_\delta(n-|\delta|)!}(n-|\lambda|)!(n-|\delta|)!}{n!\frac{n!}{z_\rho(n-|\rho|)!}}\sum_{|\rho|\leq k\leq |\lambda|+|\delta|}a_{\lambda\delta}^{\rho}(k)\frac{1}{(n-k)!},
\end{equation}
which is equal to:
\begin{equation}
\frac{z_\rho}{z_\lambda z_\delta}\sum_{|\rho|\leq k\leq |\lambda|+|\delta|}a_{\lambda\delta}^{\rho}(k)\frac{(n-|\rho|)!}{(n-k)!}.
\end{equation}

For any $|\rho|\leq k$, the quotient $\frac{(n-|\rho|)!}{(n-k)!}$ is a polynomial in $n$ with degree equals to $k-|\rho|.$
\begin{cor}
Let $\lambda$ and $\delta$ be two proper partitions. Let $n$ be an integer sufficiently big and consider the following equation :
$${\bf C}_{\underline{\lambda}_n}{\bf C}_{\underline{\delta}_n}=\sum_{\rho \text{ proper partition}} c_{\lambda\delta}^{\rho}(n) {\bf C}_{\underline{\rho}_n}.$$
The structure coefficients $c_{\lambda\delta}^{\rho}(n)$ are polynomials in $n$ with positive, rational coefficients.
\end{cor}

The polynomiality property for the structure coefficients of the center of the symmetric group algebra, given first by Farahat and Higman in 1959 in \cite{FaharatHigman1959},
is thus a direct consequence of Theorem \ref{center_mini_th}.

We show in the next example that it is possible, for some particular partitions, to give the exact values of the structure coefficients of the center of the symmetric group algebra by using Theorem \ref{center_main_th}.

\begin{ex}\label{ex:valeur_exacte_coef_S_n}
Suppose that $f$ and $h$ are both the permutation $(1\,\,2)$ of $2.$ Then in this case $k_1=k_2=2.$ For $n$ big enough, the conjugacy class in $\mathcal{S}_n$ associated to $f$ and $h$ is $C_{(2,1^{n-2})}.$ Suppose that we are looking for the coefficient of $C_{(2^2,1^{n-4})}$ in the expansion of the product $C_{(2,1^{n-2})}\cdot C_{(2,1^{n-2})}.$ In Theorem \ref{center_main_th}, the values of the sum index $k$ can be either $2,$ $3$ or $4.$ To find the coefficient of $C_{(2^2,1^{n-4})}$ in the product $C_{(2,1^{n-2})}\cdot C_{(2,1^{n-2})},$ we should first search the permutations of $\mathcal{S}_4$ which are $(2,2)$-minimal. These permutations are those which send $\lbrace 3,4\rbrace$ to $\lbrace 1,2\rbrace.$ There are $4$ such permutation : $(1\,\,3)(2\,\,4),$ $(1\,\,4\,\,2\,\,3),$ $(1\,\,3\,\,2\,\,4)$ and $(1\,\,4)(2\,\,3).$ The sum index of Theorem \ref{center_main_th} in this case consists in the following permutations : $(1\,\,4)(2\,\,3),$ $(1\,\,4\,\,2\,\,3),$ $(1\,\,3\,\,2\,\,4)$ and $(1\,\,3)(2\,\,4).$ For each $x$ among them, $xhx^{-1}f$ is the permutation $(1\,\,2)(3\,\,4)$ with cycle-type $(2^2)$ and $|\mathcal{S}_4^2 f^{-1}xh^{-1}\mathcal{S}_4^2\cap \mathcal{S}_{4}|=4.$ By Theorem \ref{center_main_th}, the coefficient which we are looking for is equal to :
$$\frac{\frac{n!}{2(n-2)!}\frac{n!}{2(n-2)!}(n-2)!(n-2)!}{n!\frac{n!}{2^2\cdot 2\cdot (n-4)!}}\cdot 4\cdot \frac{1}{(n-4)!4}=2.$$
\end{ex}

By using the same way of reasoning used in Example \ref{ex:valeur_exacte_coef_S_n}, we can find the full expression of the product ${\bf C}_{(1^{n-2},2)}^2$ given in \cite[Example 2.9]{touPhd14} :

$${\bf C}_{(1^{n-2},2)}^2=\frac{n(n-1)}{2}{\bf C}_{(1^n)}+3{\bf C}_{(1^{n-3},3)}+2{\bf C}_{(1^{n-4},2^2)}.$$

This equation agrees with the expression of the product of $A_{(2)}\cdot A_{(2)}$ given in \cite[page 4216]{Ivanov1999}, by applying the morphism $\psi$ defined in Theorem 7.1 of the same paper.

\subsection{The center of the hyperoctahedral group algebra}\label{cen_grp_hyp}

The conjugacy classes of the hyperoctahedral group $\mathcal{B}_n$ are indexed by pairs of partitions $(\lambda,\delta)$ such that $|\lambda|+|\delta|=n,$ see \cite{geissinger1978representations} or \cite{stembridge1992projective}. We start this section by giving details about this fact and describing the conjugacy classes of the hyperoctahedral group in order to define the structure coefficients of the center of the hyperoctahedral group algebra. 

It will be useful in this section to introduce the following notation: 
$$x(p(i)):=x\big(\lbrace 2i-1,2i\rbrace\big)=\lbrace x(2i-1),x(2i)\rbrace,$$
for any $x\in \mathcal{S}_{2n}$ and $1\leq i\leq n.$ By using this notation, we have :
$$\mathcal{B}_n=\lbrace x\in \mathcal{S}_{2n} \text{ such that for any $1\leq i\leq n,$ there exists $1\leq j\leq n$: }x(p(i))=p(j)\rbrace.$$
If $a\in p(i)$, we denote by $\overline{a}$ the element of the set $p(i)\setminus \lbrace a\rbrace.$ Therefore, we have, $\overline{\overline{a}}=a$ for any $a=1,\cdots 2n.$

The cycle decomposition of a permutation of $\mathcal{B}_n$ has a remarkable form. It contains two types of cycles. Suppose that $\omega$ is a permutation of $\mathcal{B}_n$ and take a cycle $\mathcal{C}$ of its decomposition, $\mathcal{C}$ can be written as follows :
$$\mathcal{C}=(a_1,\cdots ,a_{l(\mathcal{C})}),$$
where $l(\mathcal{C})$ is the length of the cycle $\mathcal{C}.$ We distinguish two cases :
\begin{enumerate}
\item first case: $\overline{a_1}$ appears in the cycle $\mathcal{C},$ for example $a_j=\overline{a_1}.$ Since $\omega\in \mathcal{B}_n$ and $\omega(a_1)=a_2,$ we have $\omega(\overline{a_1})=\overline{a_2}=\omega(a_j).$ Likewise, since $\omega(a_{j-1})=\overline{a_1},$ we have $\omega(\overline{a_{j-1}})=a_1$ which means that $a_{l(\mathcal{C})}=\overline{a_{j-1}}.$ Therefore,
$$\mathcal{C}=(a_1,\cdots a_{j-1},\overline{a_1},\cdots,\overline{a_{j-1}})$$
and $l(\mathcal{C})=2(j-1)$ is even. We will denote such a cycle by $(\mathcal{O},\overline{\mathcal{O}}).$
\item second case: $\overline{a_1}$ does not appear in the cycle $\mathcal{C}.$ Take the cycle $\mathcal{C'}$ which contains $\overline{a_1}.$ Since $\omega(a_1)=a_2$ and $\omega\in \mathcal{B}_n$, we have $\omega(\overline{a_1})=\overline{a_2}$ and so on. That means that the cycle $\mathcal{C'}$ is of the following form,
$$\mathcal{C'}=(\overline{a_1},\overline{a_2},\cdots ,\overline{a_{l(\mathcal{C})}})$$
and that $\mathcal{C}$ and $\mathcal{C'}$ appear in the cycle decomposition of $\omega.$ From now on, we will use $\overline{\mathcal{C}}$ instead of $\mathcal{C'}.$ 
\end{enumerate} 

Suppose that the cycle decomposition of a permutation $\omega$ of $\mathcal{B}_n$ is as follows:
$$\omega=\mathcal{C}_1\overline{\mathcal{C}_1}\mathcal{C}_2\overline{\mathcal{C}_2}\cdots \mathcal{C}_k\overline{\mathcal{C}_k}(\mathcal{O}^1,\overline{\mathcal{O}^1})(\mathcal{O}^2,\overline{\mathcal{O}^2})\cdots (\mathcal{O}^l,\overline{\mathcal{O}^l}).$$
Let $\lambda$ be the partition with parts the lengths of the cycles $\mathcal{C}_i,$ $i=1,\cdots, k$ and $\delta$ the partition with parts the lengths of the cycles $\mathcal{O}^j,$ $j=1,\cdots, l.$ We have, $|\lambda|+|\delta|=n.$ We define the $\textit{type}$ of $\omega$ to be the pair $(\lambda,\delta)$ of partitions.

\begin{prop}
Two permutations of $\mathcal{B}_n$ are in the same conjugacy class in $\mathcal{B}_n$ if and only if they have the same type.
\end{prop}
\begin{proof}
See section 2 in \cite{stembridge1992projective}.
\end{proof}
\begin{rem}
Two permutations of $\mathcal{B}_n$ may have the same cycle-type -- that means they may be in the same conjugacy class in $\mathcal{S}_{2n}$ -- without being in the same conjugacy class in $\mathcal{B}_n.$ For example, the permutations $\omega=(12)(34)(56)$ and $\psi=(13)(24)(56)$ of $\mathcal{B}_3$ have both the cycle-type $(2^3)$ but they are not in the same conjugacy class in $\mathcal{B}_3$ since $\omega$ has type $(\emptyset, (1^3))$ while $\psi$ has type $((2),1).$
\end{rem}

\begin{cor}
Let $\omega\in \mathcal{B}_n$ and suppose that $type(\omega)=(\lambda,\delta),$ $|\lambda|+|\delta|=n.$ Then,
$$C_\omega=\lbrace \theta\in \mathcal{B}_n \text{ such that } type(\theta)=(\lambda,\delta)\rbrace.$$
\end{cor} 
 
This shows that the conjugacy classes of the hyperoctahedral group are indexed by pairs of partitions $(\lambda,\delta)$ such that $|\lambda|+|\delta|=n$ and that for such a pair, its associated conjugacy class is :
$$\mathcal{H}_{(\lambda,\delta)}=\lbrace \theta\in \mathcal{B}_n \text{ such that } type(\theta)=(\lambda,\delta)\rbrace.$$

Let $\omega$ be a permutation of $\mathcal{B}_n$ of type $(\lambda,\delta),$ the cardinal of $\mathcal{H}_{(\lambda,\delta)}$ is :
$$|\mathcal{H}_{(\lambda,\delta)}|=\frac{|\mathcal{B}_n|}{|S_\omega|},$$
where $S_\omega=\lbrace \theta \in \mathcal{B}_n\text{ such that }\theta\omega\theta^{-1}=\omega\rbrace.$ The size of $S_\omega$ is :
$$|S_\omega|=\prod_{i\geq 1}(2i)^{m_i(\lambda)}m_i(\lambda)!\prod_{j\geq 1}(2i)^{m_j(\delta)}m_j(\delta)!=2^{l(\lambda)}z_\lambda 2^{l(\delta)}z_\delta.$$
\begin{prop}
Let $(\lambda,\delta)$ be a pair of partitions such that $|\lambda|+|\delta|=n,$ then :
$$|\mathcal{H}_{(\lambda,\delta)}|=\frac{2^nn!}{2^{l(\lambda)+l(\delta)}z_\lambda z_\delta}.$$
\end{prop}

\begin{definition}
A pair of partitions $(\lambda,\delta)$ is \textit{proper} if the partition $\lambda$ is proper. For a proper pair $(\lambda,\delta)$ of partitions and for any integer $n\geq |\lambda|+|\delta|,$ we define $(\lambda,\delta)^{\uparrow^n}$ to be the following pair of partitions with size $n$ (that means the sum of sizes of both partitions is equal to $n$) :
$$(\lambda,\delta)^{\uparrow^n}:=(\lambda\cup (1^{n-|\lambda|-|\delta|}),\delta).$$
This defines a bijection between the set of proper pairs of partitions with size less or equal to $n$ and the set of pairs of partitions with size $n.$
\end{definition} 

It is not difficult to verify that :
$$|\mathcal{H}_{(\lambda,\delta)^{\uparrow^n}}|=\frac{2^nn!}{2^{l(\lambda)+n-|\lambda|-|\delta|+l(\delta)}z_\lambda(n-|\lambda|-|\delta|)! z_\delta}=|\mathcal{H}_{(\lambda,\delta)}|\frac{n!}{(n-|\lambda|-|\delta|)!(|\lambda|+|\delta|)!}.$$

Let $(\lambda,\delta)$ and $(\beta,\gamma)$ be two proper pairs of partitions. For any integer $n\geq |\lambda|+|\delta|, |\beta|+|\gamma|,$ there exists numbers $c_{(\lambda,\delta)(\beta,\gamma)}^{(\rho,\nu)}(n)$ such that :
$$\mathcal{H}_{(\lambda,\delta)^{\uparrow^n}}\mathcal{H}_{(\beta,\gamma)^{\uparrow^n}}=\sum_{(\rho,\nu) proper\atop{|\rho|+|\nu|\leq n}} c_{(\lambda,\delta)(\beta,\gamma)}^{(\rho,\nu)}(n)\mathcal{H}_{(\rho,\nu)^{\uparrow^n}}.$$

We had already proven, in Section \ref{sec:cond_group_hyp}, that the sequence of hyporectahedral groups satisfies hypotheses H.1 to H.5. To apply Theorem \ref{center_mini_th} to the centers of the hyperoctahedral group sequence, it remains to verify if $\text{H}^{'}.0$ is satisfied for $(Z(\mathbb{C}[\mathcal{B}_n]))_n.$ This can be done by a way similar to that used to prove the same hypothesis for  $(Z(\mathbb{C}[\mathcal{S}_n]))_n.$ Consider a permutation $\omega$ of $\mathcal{B}_n.$ The conjugacy class of $\omega$ in $\mathcal{B}_n$ is the set of permutations of $\mathcal{B}_n$ which have the same type, say $(\lambda(\omega),\delta(\omega)),$ as $\omega.$ The conjugacy class of $\omega,$ seen as a permutation of $\mathcal{B}_{n+1},$ in $\mathcal{B}_{n+1}$ is the set of permutations of $\mathcal{B}_{n+1}$ with $(\lambda(\omega)\cup (1),\delta(\omega))$ as type, which is the conjugacy class of $\omega$ in $\mathcal{B}_n$ when we intersect it with $\mathcal{B}_{n+1}.$ Therefore, by Theorem \ref{center_mini_th}, there exists rational numbers $a_{(\lambda,\delta)(\beta,\gamma)}^{(\rho,\nu)}(k)$ all independent of $n$ such that :
\begin{equation*}
c_{(\lambda,\delta)(\beta,\gamma)}^{(\rho,\nu)}(n)=\frac{\frac{n!|\mathcal{H}_{(\lambda,\delta)}|2^{n-|\lambda|-|\delta|}}{(n-|\lambda|-|\delta|)!(|\lambda|+|\delta|)!}\frac{n!|\mathcal{H}_{(\beta,\gamma)}|2^{n-|\beta|-|\gamma|}}{(n-|\beta|-|\gamma|)!(|\beta|+|\gamma|)!}(n-|\lambda|-|\delta|)!(n-|\beta|-|\gamma|)!}{2^nn!|\mathcal{H}_{(\rho,\nu)}|\frac{n!}{(n-|\rho|-|\nu|)!(|\rho|+|\nu|)!}}
\end{equation*}
$$\sum_{|\rho|+|\nu|\leq k\leq \min(|\lambda|+|\delta|+|\beta|+|\gamma|,n)}a_{(\lambda,\delta)(\beta,\gamma)}^{(\rho,\nu)}(k)\frac{1}{2^{n-k}(n-k)!},$$
for three proper pairs $(\lambda,\delta),(\beta,\gamma),(\rho,\nu)$ and for any integer $n\geq |\lambda|+|\delta|,|\beta|+|\gamma|,|\rho|+|\nu|.$ After simplification, this equation can be written as follows,
\begin{eqnarray*}
c_{(\lambda,\delta)(\beta,\gamma)}^{(\rho,\nu)}(n)&=&\frac{|\mathcal{H}_{(\lambda,\delta)}||\mathcal{H}_{(\beta,\gamma)}|(|\rho|+|\nu|)!}{|\mathcal{H}_{(\rho,\nu)}|(|\lambda|+|\delta|)!(|\beta|+|\gamma|)!}\\
&&\sum_{|\rho|+|\nu|\leq k\leq \min(|\lambda|+|\delta|+|\beta|+|\gamma|,n)}a_{(\lambda,\delta)(\beta,\gamma)}^{(\rho,\nu)}(k)\frac{(n-k+1)\cdots (n-|\rho|-|\nu|)}{2^{|\lambda|+|\delta|+|\beta|+|\gamma|-k}}.
\end{eqnarray*}

\begin{cor}
Let $(\lambda,\delta), (\beta,\gamma)$ and $(\rho,\nu)$ be three proper pairs of partitions, then for any $n\geq |\lambda|+|\delta|,|\beta|+|\gamma|,|\rho|+|\nu|,$ the structure coefficient $c_{(\lambda,\delta)(\beta,\gamma)}^{(\rho,\nu)}(n)$ of the center of the hyperoctahedral group algebra is a polynomial in $n$ with non-negative coefficients and we have:
$$\deg(c_{(\lambda,\delta)(\beta,\gamma)}^{(\rho,\nu)}(n))\leq |\lambda|+|\delta|+|\beta|+|\gamma|-|\rho|-|\nu|.$$
\end{cor}

According to the author's knowledge, the polynomiality property for the structure coefficients of the center of the hyperoctahedral group algebra appears explicitly for the first time in this paper. It may be that some authors mentioned (or proved) this property before (after the polynomiality property for the structure coefficients of the center of the symmetric group algebra appeared) but the author was not able to handle any paper which mentions this result. Again, the question of using this property to study the asymptotic behaviour of some combinatorial objects related to the study of the center of the hyperoctahedral group should be asked.

\begin{que}
Does the polynomiality property for the structure coefficients of the center of the hyperoctahedral group algebra may be applied to study the asymptotic behaviour of some combinatorial objects related to this algebra ?
\end{que}

\section{Other applications and open questions}

We present in this section several interesting algebras. For some among them, there already exists a polynomiality property for the structure coefficients, for the others there isn't. The problem is that our general framework in its actual state does not contain the majority of them. An interesting research is to try to lighten the necessarily hypotheses in our general framework in order to include more of these algebras, while still allowing the possibility to obtain a theorem for the polynomiality property. The reader can have a look to the paper \cite{strahov2007generalized} written by Strahov where the author gives a list of interesting double-class algebras. These algebras are well known and studied in the literature and a polynomiality property for their structure coefficients would be interesting.

\subsection{The center of the group of invertible matrices with elements in a finite field algebra}\label{sec:centre_Gl_n}

We recall in this section the work of Méliot \cite{meliot2013partial} about the structure coefficients of the center of the group $GL_n(\mathbb{F}_q)$  algebra, where $GL_n(\mathbb{F}_q)$ is the group of invertible $n\times n$-matrices with coefficients in the finite field $\mathbb{F}_q$ with $q$ elements. We use the same notations used by Méliot.  

The center of $GL_n(\mathbb{F}_q)$ algebra is linearly generated by classes $C_{\hat{\mu}}$ indexed by the polypartitions of size $n$ over the finite field $\mathbb{F}_q.$ A polypartition $\hat{\mu}=\lbrace \mu(P_1),\cdots,\mu(P_r) \rbrace$ of size $n$ over $\mathbb{F}_q$ is a family of partitions indexed by monic irreducible polynomials over $\mathbb{F}_q$, all different from the polynomial $X$, such that :

$$|\hat{\mu}|=\sum_{i=1}^r \deg(P_i)\mid\mu(P_i)\mid.$$   

The size of $GL_n(\mathbb{F}_q)$ is $(q^n-1)(q^n-q)\cdots (q^n-q^{n-1}).$ As given in \cite[section 1.2]{meliot2013partial}, for a fixed polypartition $\hat{\mu}$ of size $n$ over $\mathbb{F}_q,$ the size of its associated class $C_{\hat{\mu}}$ is :
\begin{equation}
|C_{\hat{\mu}}|=\frac{(q^n-1)(q^n-q)\cdots (q^n-q^{n-1})}{q^{|\hat{\mu}|+2b(\hat{\mu})}\prod_{i=1}^{r}\prod_{k\geq 1}(q^{-\deg P_i})_{m_k(\mu(P_i))}},
\end{equation}
where $$b(\hat{\mu})=\sum_{i=1}^r (\deg P_i)b(\mu(P_i))=\sum_{i=1}^r\sum_{j=1}^{l(\mu(P_i))}(\deg P_i)(j-1)(\mu(P_i))_j,$$
$(x)_m=(x;x)_m$ is the Pochhammer symbol $(1-x)(1-x^2)\cdots (1- x^m)$ and $m_k(\mu)$ is the number of $k$-parts in the partition $\mu.$ 

\begin{definition}
We say that a polypartition $\hat{\mu}$ is proper if the partition $\mu(X-1)$ is proper. 
\end{definition}

The set of polypartitions of $n$ is in bijection with the set of proper polypartitions with size less or equal to $n.$ We index the basis of $Z(\mathbb{C}[GL_n(\mathbb{F}_q)])$ by proper polypartitions in order to present the polynomiality property of its structure coefficients.

To any proper polypartition $\hat{\mu}$ of size $k$ less or equal to $n$, we associate a polypartition of $n$ which we denote $\hat{\mu}^{\uparrow n}.$ The partition $\mu^{\uparrow n}(X-1)=\mu(X-1)\cup (1^{n-k})$ while the other partitions of $\hat{\mu}^{\uparrow n}$ associated to the other irreducible polynomials are the same as for $\hat{\mu}.$ It is not difficult to verify that :
$$|\hat{\mu}^{\uparrow n}|=\sum_{i=1}^r \deg(P_i)|\mu^{\uparrow n}(P_i)|=\sum_{i=1}^r \deg(P_i)|\mu(P_i)|+(n-k)=n,$$
and
\begin{eqnarray*}
2b(\hat{\mu}^{\uparrow n})&=&2\sum_{i=1}^r\sum_{j=1}^{l(\mu^{\uparrow n}(P_i))}(\deg P_i)(j-1)(\mu^{\uparrow n}(P_i))_j\\
&=&2b(\hat{\mu})+2\sum_{j=l(\mu(X-1))+1}^{l(\mu(X-1))+n-|\hat{\mu}|}j-1\\
&=&2b(\hat{\mu})+\big(n-|\hat{\mu}|\big)\big(n-|\hat{\mu}|+2l(\mu(X-1))-1\big).
\end{eqnarray*}

Therefore, if $\hat{\mu}$ is a polypartition with size less or equal to $n$, we get :

\begin{eqnarray*}
|C_{\hat{\mu}^{\uparrow n}}|&=&\frac{(q^n-1)(q^n-q)\cdots (q^n-q^{n-1})}{q^{n+2b(\hat{\mu})+\big(n-|\hat{\mu}|\big)\big(n-|\hat{\mu}|+2l(\mu(X-1))-1\big)}\prod_{i=1}^{r}\prod_{k\geq 1}(q^{-\deg P_i})_{m_k(\mu(P_i))}(q^{-1})_{n-|\hat{\mu}|}}\\
&=&|C_{\hat{\mu}}|\frac{|GL_n(\mathbb{F}_q)|}{|GL_{|\hat{\mu}|}(\mathbb{F}_q)|q^{\big(n-|\hat{\mu}|\big)\big(n-|\hat{\mu}|+2l(\mu(X-1))\big)}(q^{-1})_{n-|\hat{\mu}|}}\\
&=&|C_{\hat{\mu}}|\frac{|GL_n(\mathbb{F}_q)|}{|GL_{|\hat{\mu}|}(\mathbb{F}_q)|q^{\big(n-|\hat{\mu}|\big)\big(2l(\mu(X-1))\big)}|GL_{n-|\hat{\mu}|}(\mathbb{F}_q)|}.
\end{eqnarray*}

The last equality comes from the fact that :

$$(q^{-1})_n=\frac{|GL_{n}(\mathbb{F}_q)|}{q^{n^2}}.$$

The following theorem is Theorem 3.7 in \cite{meliot2013partial} about the polynomiality property for the structure coefficients of $Z(GL_n(\mathbb{F}_q)).$ Our presentation here is a little bit different from Méliot's one since we use the proper polypartition to be more consistent with the results of polynomiality already presented in this paper. 

\begin{theoreme}[Méliot]\label{Th:Meliot}
Let us fix $q$ and let $\hat{\lambda}$, $\hat{\delta}$ be two proper polypartitions. Let $n$ be an integer sufficiently big and consider the following equation :
\begin{equation}
\mathcal{C}_{\hat{\lambda}^{\uparrow n}}\mathcal{C}_{\hat{\delta}^{\uparrow n}}=\sum_{\hat{\rho} \text{ proper polypartition}} c_{\hat{\lambda}\hat{\delta}}^{\hat{\rho}}(n) \mathcal{C}_{\hat{\rho}^{\uparrow n}}.\end{equation}
Then the coefficients $c_{\hat{\lambda}\hat{\delta}}^{\hat{\rho}}(n)$ are polynomials in $q^n$ with rational coefficients.
\end{theoreme}

\begin{ex}\label{ex:coeff_meliot}
Let $a$ be an element of $\mathbb{F}^*_q$ and consider its inverse $a^{-1}.$ Let $\hat{\lambda}_a$ (resp. $\hat{\delta}_{a^{-1}}$) be the polypartition where the partition associated to the irreducible polynomial $X-a$ (resp. $X-a^{-1}$) is $(1)$ and the empty partition is associated to all of the other irreducible polynomials. It is evident that $\hat{\lambda}_a$ and $\hat{\delta}_{a^{-1}}$ are proper polypartitions if $a$ is different from $1.$ Let $n$ be an integer big enough and denote by $\hat{\emptyset}$ the proper polypartition which associates to all of the irreducible polynomials the empty partition. We are interested in the coefficient $\mathcal{C}_{\hat{\emptyset}^{\uparrow n}}$ in the product $\mathcal{C}_{\hat{\lambda}_a^{\uparrow n}}\mathcal{C}_{\hat{\delta}_{a^{-1}}^{\uparrow n}}.$ The identity matrix of size $n$, $I_n$, is the only element of $\mathcal{C}_{\hat{\emptyset}^{\uparrow n}}$, while the elements of $\mathcal{C}_{\hat{\lambda}_a^{\uparrow n}}$ (resp. $\mathcal{C}_{\hat{\delta}_{a^{-1}}^{\uparrow n}}$) are the matrices conjugated to $I_{a,n-1}$ (resp. $I_{a^{-1},n-1}$) where :
$$I_{a,n-1}=\begin{pmatrix}
a&0&0&\cdots &0\\
0&1&0&\cdots &0\\
\vdots &0 &\ddots &\cdots & \vdots\\
0&0&\cdots &\ddots &0\\
0&0&\cdots &\cdots &1
\end{pmatrix} \text{(resp.} I_{a^{-1},n-1}=\begin{pmatrix}
a^{-1}&0&0&\cdots &0\\
0&1&0&\cdots &0\\
\vdots &0 &\ddots &\cdots & \vdots\\
0&0&\cdots &\ddots &0\\
0&0&\cdots &\cdots &1
\end{pmatrix}).$$
Let $A$ be a matrix conjugated to $I_{a,n-1}$, say $A=MI_{a,n-1}M^{-1}$ for a matrix $M\in GL_n(\mathbb{F}_q)$, then $B=MI_{a^{-1},n-1}M^{-1}$ is conjugated to $I_{a^{-1},n-1}$ and $AB=I_n.$ That shows that the coefficient we are looking for is equal to the size of the conjugacy class of  $I_{a,n-1}$ which is :
$$q^{n-1}\frac{q^{n}-1}{q-1}=\frac{1}{q(q-1)}\big((q^n)^2-q^n\big).$$
\end{ex}

Unfortunately, our general framework does not contain the case of the center of the $GL_n(\mathbb{F}_q)$ group algebra. In fact, the sequence of $GL_n(\mathbb{F}_q)$ does not satisfy the fourth hypothesis H.4 given in Section \ref{sec:hypo_defi} and thus Theorem \ref{center_mini_th} could not be applied. Below, we give an explicit counter-example :

\begin{c-ex}\label{contre_exemple_GL}[to H.4 in the case of $GL_n(\mathbb{F}_q)$]
For $k_1=2$ and $k_2 = 1,$ the $5\times 5$ following matrix :
$$\begin{pmatrix}
1&1&1&0&0\\
1&0&0&0&0\\
0&1&0&0&1\\
0&0&1&0&0\\
0&0&0&1&0
\end{pmatrix}$$
can not be made to be a matrix with the following form
$$\begin{pmatrix}
*&*&*&0&0\\
*&*&*&0&0\\
*&*&*&0&0\\
0&0&0&1&0\\
0&0&0&0&1
\end{pmatrix}$$
by using the elementary operations on the three last columns and the lines 2 to 5. These elementary operations are equivalent as to see the left (right) class of $GL_5(\mathbb{F}_q)^2$ ($GL_5(\mathbb{F}_q)^1$) where $GL_n(\mathbb{F}_q)^k$ is the sub-group in $GL_n(\mathbb{F}_q)$ of matrices with the following form :

$$ \left(\begin{array}{c|c}
  \begin{array}{cccccc}
    I_k
  \end{array}

 & \mbox{\Huge $0$}\\
 
 \hline
 
 \begin{matrix}
 \mbox{\Huge $0$}

\end{matrix}
 &
 
 \begin{matrix}
 * & * &\cdots &*\\
 * &\ddots &\cdots &* \\
 \vdots &\ddots &\cdots &\vdots \\
 * &\cdots &\cdots &*

\end{matrix}
 
 \end{array}\right)$$

\end{c-ex}

\begin{rem}
Since we present here the first case of interesting algebras where our general framework can not be used, it will be fair to clarify the following point. The reader must have remarked in Counter-example \ref{contre_exemple_GL} to H.4 in the case of $GL_n(\mathbb{F}_q)$ that we tried "one" sub-groups family with elements $GL_n(\mathbb{F}_q)^k$ for which hypothesis H.4 is not satisfied. However, which we demand, in our general framework, is the "existence", for any $k,$ of a sub-group $K_n^k$ which verifies the necessarily hypotheses H.0 to H.5. Since we have just tried one (and not all possible) sub-group of $GL_n(\mathbb{F}_q),$ in Counter-example \ref{contre_exemple_GL}, we can not "directly" say that our general framework can not be applied in the case of $Z(GL_n(\mathbb{F}_q)).$ Which makes us almost-sure that our general framework does not contain the case of the sequence of $GL_n(\mathbb{F}_q)$ is the fact that the sub-groups $GL_n(\mathbb{F}_q)^k$ which we tried are the natural ones in this case and they are similar (in this case) to the sub-groups $\mathcal{S}_n^k$ and $\mathcal{B}_n^k$ (which were the good choices) in the case of symmetric and hyperoctahedral groups.

We will use the same logic in the next cases of algebras where we can not apply our general framework. When we give a counter-example, that will mean that the sub-groups which we are using in that case are the most natural to try. 
\end{rem}

Note that hypothesis H.4 given in Section \ref{sec:hypo_defi} is the most important in our reasoning. It allows us to have a representative of the class $K_n^{k_1}xK_n^{k_2}$ in $K_{k_1+k_2}$ and it allows the index $k,$ in the sums of our principal theorems, to be bounded by an integer which does not depend on $n,$ which is crucial to obtain polynomials.

Since we can not apply Theorems \ref{main_th} and \ref{mini_th} in the case of the sequence of $GL_n(\mathbb{F}_q),$ it will be natural to ask the following question :
\begin{que}\label{ques:GL}
Is it possible to modify the hypotheses in our general framework, while still allowing the possibility to obtain theorems similar to Theorem \ref{main_th} and Theorem \ref{mini_th}, in order to include the case of the sequence of the centers of the $GL_n(\mathbb{F}_q)$ groups algebras in our list of applications ?
\end{que}

\subsection{Super-classes and double-classes}

The theory of super-characters and super-classes of group algebras is studied with details in \cite{diaconis2008supercharacters} by Diaconis and Isaacs. By definition, the super-classes are unions of conjugacy classes. In this section, we show that the super-classes of a finite group $G$ (of the form $1+J,$ where $J$ is a nilpotent algebra) are in bijection with the double-classes of a particular pair of groups. In the next section we consider the case where $G$ is the group of uni-triangular matrices.

Let $J$ be a nilpotent\footnote{An algebra $\mathcal{A}$ is nilpotent if there exists an integer $n$ such that $x^n=0$ for any $x\in \mathcal{A}$} and associative algebra with finite dimension over a finite field. We consider the set $G=1+J$ (formal sum) of elements which are written in the form $1+x$ where $x\in J.$ The set $G$ is a group with law defined as follows:
$$(1+x)(1+y)=1+x+y+xy,$$
for any $x,y\in J.$ The direct product group $H=G\times G$ acts on $J$ by :
$$(u,v)\cdot x=uxv^{-1},$$
where $x\in J$ and $(u,v)\in G\times G.$ A $\textit{super-class}$ of $G$ is an orbit of this action, that means a set of the form $1+(G,G)\cdot x,$ for a certain $x\in J$, where :
$$1+(G,G)\cdot x=\lbrace 1+uxv^{-1} \, ; \, u,v\in G\rbrace.$$

We consider the semi-direct product of groups $(G\times G)$ and $J$ denoted by $(G\times G)\ltimes J.$ As a set, $(G\times G)\ltimes J$ is the set of elements of the form $((u,v),x),$ where $u,v\in G$ and $x\in J.$ The product in $(G\times G)\ltimes J$ is defined as follows :
$$((u,v),x)\cdot ((u',v'),x')=((uu',vv'), x+ux'v^{-1}),$$
for any elements $((u,v),x), ((u',v'),x')\in (G\times G)\ltimes J.$ The set of elements of the form $((u,v),0)$ is a sub-group of $(G\times G)\ltimes J$ isomorphic to $G\times G.$ For two elements $(g_1,g_2)$ and $(g'_1,g'_2)$ of $G\times G,$ and for any element $((u,v),x)$ of $(G\times G)\ltimes J,$ we have :
$$(g_1,g_2)\cdot ((u,v),x) \cdot (g'_1,g'_2):=((g_1,g_2),0)\cdot ((u,v),x) \cdot ((g'_1,g'_2),0)=((g_1ug'_1,g_2vg'_2),g_1xg_2^{-1}).$$
There is a bijection between the super-classes of $G$ and the double-classes of $G\times G$ in $(G\times G)\ltimes J.$ For an element $x\in J,$ all the elements of the form $((u,v),x)\in (G\times G)\ltimes J$ are in the same $G\times G$ double-class in $(G\times G)\ltimes J$ which is the image of the super-class of $x$ by this bijection. Explicitly, the function $\psi$ defined as follows :
$$\begin{array}{ccccc}
\psi & : & \mathcal{SC}(G) & \longrightarrow & G\times G\setminus(G\times G)\ltimes J/ G\times G \\
& & 1+(G,G)\cdot x & \mapsto & (G\times G)\times GxG\\
\end{array},$$
where $\mathcal{SC}(G)$ is the set of super-classes of $G,$ is a bijection. In what follows, for an element $x\in J$, we will denote by $\mathcal{SC}(x)$ (resp. $\mathcal{DC}(x)$) the super-class (resp. double-class) of $G$ (resp. $G\times G$ in $(G\times G)\ltimes J$) associated to $x.$ By linearity, $\psi$ can be extended over the algebra of super-classes of $G.$ To have a morphism of algebras, we should consider $\frac{\psi}{|G|^2}.$ In other words, we have the following proposition.
\begin{prop}
The function $\frac{1}{|G|^2}\psi:\mathbb{C}[\mathcal{SC}(G)]\rightarrow \mathbb{C}[G\times G\setminus(G\times G)\ltimes J/ G\times G]$ defined on the basis elements by: 
$$(\frac{1}{|G|^2}\psi)(\mathcal{SC}(x)):=\frac{1}{|G|^2}\psi(\mathcal{SC}(x))=\frac{1}{|G|^2}\mathcal{DC}(x),$$
is a morphism of algebras.
\end{prop}
\begin{proof}
We will prove that $\frac{1}{|G|^2}\psi$ is compatible with the products. That means that
\begin{equation}\label{ex:psi_morphisme}
\frac{1}{|G|^2}\psi(\mathcal{SC}(x)\mathcal{SC}(y))=\frac{1}{|G|^2}\psi(\mathcal{SC}(x))\frac{1}{|G|^2}\psi(\mathcal{SC}(y)), \text{ for any $x,y\in J$.}
\end{equation} 
Let $x,y$ and $z$ be three fixed elements of $J,$ we denote by $s_{xy}^{z}$ (resp. $d_{xy}^{z}$) the coefficient of $\mathcal{SC}(z)$ (resp. $\mathcal{DC}(z)$) in the expansion of the product $\mathcal{SC}(x)\mathcal{SC}(y)$ (resp. $\mathcal{DC}(x)\mathcal{DC}(y)$). Proving Equation \eqref{ex:psi_morphisme} is equivalent to prove that $d_{xy}^{z}=|G|^2s_{xy}^{z}.$ By using the direct way of computing structure coefficients (classical, see for example \cite[Proposition 2.1]{toutAFrobe}) we can write :
$$s_{xy}^{z}=|\lbrace (a,b,c,d)\in G^4 \text{ such that } axb+cyd+axbcyd=z\rbrace|=:|A_{xy}^{z}|,$$
and
$$d_{xy}^{z}=|\lbrace (a,b,c,d,e,f)\in G^6 \text{ such that } axb+cdyef=z\rbrace|$$$$=|G|^2|\lbrace (a,b,c',d')\in G^4 \text{ such that } axb+c'yd'=z\rbrace|.$$
If we define $B_{xy}^{z}$ to be the following set :
$$B_{xy}^{z}:=\lbrace (a,b,c',d')\in G^4 \text{ such that } axb+c'yd'=z\rbrace,$$
then the function defined on $A_{xy}^{z}$ with values in $B_{xy}^{z}$ which for an element $(a,b,c,d)$ return $(a,b,c(1+axb),d)$ is a bijection with inverse the function $$(a,b,c,d)\mapsto (a,b,c(1+axb)^{-1},d).$$ That means that $d_{xy}^{z}=|G|^2s_{xy}^{z}$ and that $\frac{1}{|G|^2}\psi$ is a morphism.
\end{proof}

In appendix B in \cite{diaconis2008supercharacters}, Diaconis and Isaacs present the link between super-characters of $G$ and the zonal spherical functions of the pair $\big((G\times G)\ltimes J,G\times G\big).$ In this section, we have explicitly showed the link between the super-classes of $G$ and the double-classes of the pair $\big((G\times G)\ltimes J,G\times G\big).$ It is probably possible to pass from one of these results to the another by using \cite[Proposition 1.47]{touPhd14} and an equivalent result for super-characters (in the case such a result exists). However, we preferred to present here a direct proof for the equivalence between the algebra of super-classes of $G$ and the double-class algebra of the pair $\big((G\times G)\ltimes J,G\times G\big).$

\subsection{The super-classes of the group of uni-triangular matrices}\label{super-classe-matr-uni}

The theory of super-characters and super-classes of the group of uni-triangular matrices algebra is in relation with the theory of symmetric functions with non-commutative variables, see \cite{andre2013supercharacters}. We start this section by defining the super-classes of uni-triangular matrices group and then we explain this relation at the end of this section once are given all necessary information to present it.

Let $\mathbb{K}$ be a finite field of order $q.$ For any $n\in \mathbb{N}$, we denote by $U_n$ the group of upper uni-triangular matrices with coefficients in $\mathbb{K}.$ If we denote by $\mathbf{u}_n$ the algebra over $\mathbb{K}$ of strictly upper triangular matrices with coefficients in $\mathbb{K}$, then we have, $U_n=I_n+\mathbf{u}_n$, where $I_n$ is the identity matrix of size $n.$

Let $n$ be a non-negative integer, we define $[[n]]:=\lbrace (i,j):1\leq i<j\leq n\rbrace.$
A set partition $\pi$ of $n$, written $\pi=B_1/B_2/\cdots /B_l$, is a family of non-empty sets $B_i$ such that $B_1\sqcup B_2\sqcup\cdots \sqcup B_l=[n].$ The $B_i$'s are called blocks of $\pi$ and $l(\pi)$ is the number $l$ of these blocks. Note that we are not interested in the order of the blocks of a set partition. For example :
$$\sigma=1~~2~~5/3/4~~8~~9~~10/6~~7$$
is a set partition of $10$ and $l(\sigma)=4.$ We denote by $\mathcal{SP}_n$ the set of set partitions of $n.$

By convention, we will always put the elements of a block $B$ in an increasing order, and for a block $B=b_1b_2\cdots b_k$ we associate a set of arcs, denoted $D(B)$, 
$$D(B):=\lbrace (b_1,b_2),(b_2,b_3),\cdots, (b_{k-1},b_k)\rbrace.$$
The set of arcs of a set partition $\pi$, denoted $D(\pi)$ is the disjoint union of the sets of arcs of the blocks of $\pi.$ For example : $$D(\sigma)=\lbrace (1~~2),(2~~5),(4~~8),(8~~9),(9~~10),(6~~7)\rbrace.$$
It is evident that for a set partition $\pi$ of $n$, $D(\pi)\subset [[n]].$ The inverse is not always true. That means that there are sub-sets of $[[n]]$ which do not correspond to any set partition of $n.$ For example, $D=\lbrace (1,2),(1,3)\rbrace\subset [[3]]$ could not be the set of arcs of any set partition of $3.$

For a set partition $\pi$ of $n,$ we can associate an $n\times n$ (strictly) upper triangular matrix, denoted $M(\pi),$ with entries the integers $0$ and $1.$ The matrix $M(\pi)$ is coded by the elements of $D(\pi).$ The entry $m_{ij}$ is $1$ if the arc $(i,j)$ is an element of $D(\pi)$ and $0$ if-not.

Let $\mathbb{K}^*=\mathbb{K}\setminus\lbrace 0\rbrace$, a $\mathbb{K}^*$-coloured set partition of $n$ is a pair $(\pi,\phi)$, where $\pi$ is a set partition of $n$ and $\phi:D(\pi)\rightarrow \mathbb{K}^*$ is a function. We will write $(\pi,\phi)=((a_1,\alpha_1),(a_2,\alpha_2),\cdots,(a_r,\alpha_r))$, where $D(\pi)=\lbrace a_1,a_2,\cdots ,a_r\rbrace$ and $\alpha_i=\phi(a_i)$, $1\leq i\leq r.$ For a $\mathbb{K}^*$-coloured set partition of $n,$ we can associate a $n\times n$ (strictly) upper triangular matrix, denoted $M(\pi,\phi),$ with entries in $\mathbb{K}^*.$ The matrix $M(\pi,\phi)$ has the same form as $M(\pi)$ with entry $\alpha(i,j)$ (instead of $1$) if the arc $(i,j)$ is in the set $D(\pi).$ 

Let $\pi$ be a set partition of $k$ and let $n$ be an integer greater than $k.$ We can make, in a natural way, a set partition of $n$ using $\pi$ by adding the $n-k$ blocs $k+1/k+2/\cdots/ n$ to $\pi.$ We denote this set partition by $\pi^{\uparrow n} :$
$$\pi^{\uparrow n}:=\pi/k+1/k+2/\cdots/ n.$$
In term of matrices, $M(\pi^{\uparrow n})$ is the $n\times n$ (strictly) upper triangular matrix obtained from $M(\pi)$ by adding $n-k$ $0$-columns and $0$-lines to $M(\pi).$

We say that a set partition $\pi$ of $n$ is \textit{proper} if $n$ is not alone in its block in $\pi.$ For example, the set partition $156/237/4$ of $7$ is proper while $156/7/234$ is not. We denote by $\mathcal{PSP}_n$ the set of proper set partitions of $n.$ There is a natural bijection between $\mathcal{SP}_n$ and the set $\mathcal{PSP}_{\leq n},$ 
$$\mathcal{PSP}_{\leq n}:=\bigsqcup_{1\leq k\leq n} \mathcal{PSP}_k .$$  

The super-classes of the uni-triangular group are indexed by $\mathbb{K}^*$-coloured set partitions, see \cite{andre2013supercharacters} for more details about the theory of super-characters and super-classes of the uni-triangular group. For each element $(\pi,\phi)$ of $\mathcal{SP}_n(\mathbb{K})$, the set of $\mathbb{K}^*$-coloured set partitions of $n$, we denote by $\mathcal{O}_{\pi,\phi}$ the $U_n$-double-class $U_nM(\pi,\phi)U_n$ and by $\mathcal{K}_{\pi,\phi}=I_n+\mathcal{O}_{\pi,\phi}$ the super-class of $U_n$ associated to $(\pi,\phi).$

Let $(\pi,\phi)$ and $(\sigma,\psi)$ be two proper $\mathbb{K}$-coloured set partitions of $k_1$ et $k_2$ respectively and let $n$ be an integer greater than $k_1$ and $k_2,$ then we have :
\begin{equation}
\mathcal{K}_{\pi,\phi}(n)\mathcal{K}_{\sigma,\psi}(n)=\sum_{(\rho,\theta)\in \mathcal{PSP}_{\leq n}}d_{(\pi,\phi),(\sigma,\psi)}^{(\rho,\theta)}(n)\mathcal{K}_{\rho,\theta}(n).
\end{equation}

\begin{que}
Do the coefficients $d_{(\pi,\phi),(\sigma,\psi)}^{(\rho,\theta)}(n)$ have a polynomiality property in $q^n$ ?
\end{que}

As in the case of $GL_n(\mathbb{F}_q),$ H.4 is not satisfied for the sequence of $U_n$ and probably\footnote{We suppose, to come to this conclusion, that the sub-groups $(U_n\times U_n)^{k}$ which we are looking for are of the form $(U_n)^{k} \times (U_n)^{k},$ which appears reasonable to us in this case. But that does not prove that we can not find sub-groups of $U_n \times U_n$ which are not cartesian products of sub-groups of $U_n$ with themselves and for which hypothesis H.4 is satisfied. We should also mention, that in the general case, if a sequence of subgroups $K_n$ satisfies hypotheses H.1 to H.5 then the sequence of subgroups $K_n\times K_n$ satisfies these same hypotheses. Just take $(K_n\times K_n)^k$ to be $K_n^k\times K_n^k$ for any $k\leq n.$} not for $U_n\times U_n$ (because $\mathrm{k}( (U_n\times U_n)^{k_1} (x,y) (U_n\times U_n)^{k_2}) = \max 
(\mathrm{k}(U_n^{k_1}xU_n^{k_2}), \mathrm{k}(U_n^{k_1}yU_n^{k_2})$). So it is interesting to re-ask the Question \ref{ques:GL} in order to include this case also in our general framework.

As mentioned, the study of the super-classes of the group of uni-triangular matrices is in relation with the theory of symmetric functions in non-commuting variables. In fact, the vector space $$\SC:=\bigoplus_n \SC_n,$$ where $\SC_n$ is the vector space generated by the super-characters of the uni-triangular group $U_n,$ is isomorphic "as a Hopf algebra" to the algebra of symmetric functions in non-commuting variables called $\NCSym,$ see \cite[Section 4.4]{andre2013supercharacters}. The symmetric functions in non-commuting variables are studied by Wolf in \cite{wolf1936symmetric}. 

The algebra $\NCSym$ introduced by Rosas and Sagan in \cite{rosas2006symmetric} can be viewed as an extension of the algebra $\Lambda$ of symmetric functions. It has many basis families, indexed by set partitions, similar to that of power functions, monomial functions, elementary functions, etc. To illustrate this, we take for example the monomial functions. If $\pi\in \mathcal{SP}_n,$ a monomial of the form $\pi$ in non-commuting variables is a product $x_{i_1}x_{i_2} \cdots x_{i_n} $ where
$i_r = i_s$ if and only if $r$ and $s$ are in the same bloc of $\pi.$ For example, $x_1x_2x_1x_2$ is a monomial of the form $1\,\,3/2\,\,4$ in non-commuting variables. If $\pi$ is a set partitions, the monomial symmetric function in non-commuting variables $m_\pi$ is defined to be the sum of all monomials in non-commuting variables of the form $\pi.$ For example,
$$m_{1\,\,3/2\,\,4}=x_1 x_2 x_1 x_2 + x_2 x_1 x_2 x_1 + x_1 x_3 x_1 x_3 + x_3 x_1 x_3 x_1 + x_2 x_3 x_2 x_3 + \cdots.$$
The family $(m_\pi)$ indexed by set partitions forms a basis for $\NCSym.$ For more details about this algebra, the reader can see \cite{wolf1936symmetric}, \cite{rosas2006symmetric}, \cite{gebhard2001chromatic} and \cite{andre2013supercharacters}.

\subsection{Generalisation of the Hecke algebra of the pair $(\mathcal{S}_{2n},\mathcal{B}_n)$}

Let $n$ and $k$ be two positive integers. We consider the symmetric group $\mathcal{S}_{kn}.$ We denote by $\mathcal{B}_{kn}^k$ the following set :
$$\mathcal{B}_{kn}^k:=\lbrace \sigma\in \mathcal{S}_{kn} \text{ such that for each $1 \leq r \leq n$ there exists $1 \leq r' \leq n$: } \sigma(p_k (r)) =p_k( r' )\rbrace,$$
where $p_k(i)= \lbrace (i-1)k+1, .....,  ik \rbrace.$ The set $\mathcal{B}_{kn}^k$ is a sub-group of $\mathcal{S}_{kn}$ for each $n$ and $k.$ With these notations, the pair $(\mathcal{S}_{2n},\mathcal{B}_n)$ is none other than the pair $(\mathcal{S}_{2n},\mathcal{B}_{2n}^2).$

\begin{que}\label{ques1_B}
How to define the 'type' of a permutation $\sigma$ of $\mathcal{S}_{kn}$ such that two permutations are in the same double-class $\mathcal{B}_{kn}^k\sigma \mathcal{B}_{kn}^k$ if and only if they both have the same 'type' as $\sigma$ ?
\end{que}

\begin{que}\label{ques2_B}
Is the double-class algebra of $\mathcal{B}_{kn}^k$ in $\mathcal{S}_{kn}$ commutative ?
\end{que}

\begin{que}\label{ques3_B}
Does the double-class algebra of $\mathcal{B}_{kn}^k$ in $\mathcal{S}_{kn}$ enter in our general framework ? In other way, is it possible to give a polynomiality property for the structure coefficients of this algebra by using Theorem \ref{mini_th} ?
\end{que}

If the response to Question \ref{ques3_B} is positive, we may put the following additional question :

\begin{que}\label{ques4_B}
Is the partial elements algebra associative in this case ? If not, is it possible to build a similar but associative algebra which plays the same role ? In addition, is there any relation between this algebra and the algebra of symmetric functions ?
\end{que}

Otherwise, that means if the response to Question \ref{ques3_B} is negative, we may ask the following question :

\begin{que}\label{ques5_B}
Could-we adapt the approach of the author's paper \cite{toutejc} to establish a polynomiality property in this case ?
\end{que}



An interesting sub-group of $\mathcal{B}_{kn}^k$ is that of permutations of $nk$ coloured with $k$ colours. This group has different equivalent definitions, see \cite{bagno2006excedance}. Here we use the definition which is coherent with our work and notations and we denote this group by $\mathcal{C}_{kn}^k.$

For each $1\leq i \leq n,$ we denote by $c^k_i$ the following cycle of length $k$ :
$$c_i^k:=( (i - 1)k + 1, \cdots , ik ).$$ 
We say that a permutation $\omega$ of $nk$ is $k$-coloured if $\omega$ sends every cycle $c_i^k$ to another cycle with the same form. For example :
$$p=\begin{pmatrix}
1&2&3&4&5&6&7&8&9&10&11&12\\
6&4&5&11&12&10&8&9&7&1&2&3
\end{pmatrix}\in \mathcal{C}_{12}^3,$$
but
$$p'=\begin{pmatrix}
1&2&3&4&5&6&7&8&9&10&11&12\\
6& \color{red}{5}&\color{red}{4}&11&12&10&8&9&7&1&2&3
\end{pmatrix},$$
is in $\mathcal{B}_{12}^3$ without being in $\mathcal{C}_{12}^3$ because the images of $1,$ $2$ and $3$ are not correctly (cyclically) ordered. The group $\mathcal{C}_{kn}^k$ is the set of $k$-coloured permutations of $nk,$ 
$$\mathcal{C}_{kn}^k:=\lbrace \omega\in \mathcal{S}_{nk} \text{ such that for each $1 \leq i \leq n$ there exists $1 \leq i' \leq n$: } \omega(c_i^k)=c_{i'}^k, \,\, 1\leq i,i'\leq n\rbrace.$$

It is also interesting to consider Questions \ref{ques1_B}, \ref{ques2_B}, \ref{ques3_B}, \ref{ques4_B} and \ref{ques5_B}, by replacing $\mathcal{B}_{kn}^k$ by $\mathcal{C}_{kn}^k.$

\subsection{The Iwahori-Hecke algebra and its center}\label{sec:alg_Iwahori_Hecke}

The Iwahori-Hecke algebra, denoted $\mathcal{H}_{n,q}$ is an algebra over $\mathbb{C}(q)$ which generalises the symmetric group algebra. When $q=1,$ this algebra is the symmetric group algebra $\mathbb{C}[\mathcal{S}_n].$ According to a result of Iwahori, see \cite{iwahori1964structure}, when $q$ is a power of a prime number, this algebra is isomorphic to the double-class algebra of superior triangular matrices group in $GL_n(\mathbb{F}_q).$ This result can also be found in \cite[Section 8.4]{geck2000characters}.

The Geck-Rouquier elements defined in \cite{geck1997centers} form a basis for the center $Z(\mathcal{H}_{n,q})$ of the Iwahori-Hecke algebra. They are indexed by partitions of $n$ and usually denoted by $\Gamma_{\lambda,n}.$ In \cite{francis1999minimal}, Francis gives a characterisation for these elements. The Geck-Rouquier elements $\Gamma_{\lambda,n}$ become conjugacy classes when $q=1.$

A polynomiality property for the structure coefficients $a_{\lambda\delta}^\rho(n)$ of the center of the Iwahori-Hecke algebra defined by the following equation :
\begin{equation}
\Gamma_{\lambda,n}\Gamma_{\delta,n}=\sum_{|\rho|\leq |\lambda|+|\delta|}a_{\lambda\delta}^\rho(n)\Gamma_{\rho,n}
\end{equation}
was conjectured by Francis and Wang in \cite{francis1992centers}. In \cite{meliot2010products}, Méliot proves this result.

The author can not find in the literature a result which describes the center of the Iwahori-Hecke algebra $\mathcal{H}_{n,q}$ (and not the algebra $\mathcal{H}_{n,q}$ itself) as a double-class algebra in order to see whether or not our general framework presented in this paper can contain the result of polynomiality for the structure coefficients given by Méliot.

\begin{que}
Is the center of the Iwahori-Hecke algebra a double-class algebra ? If yes, does our general framework include it ? (that means: can we re-obtain the polynomiality property of its structure coefficients, given by Méliot, using Theorem \ref{main_th} ?)  
\end{que}

\section*{Acknowledgement}
I started thinking about this generalisation while I was a Phd student under the supervision of Jean-Christophe Aval and Valentin Féray. They both encouraged and helped me to build the general framework presented here. Without them, I would never had the courage to take this challenge. So I would like to thank them for this.

\bibliographystyle{alpha}
\bibliography{biblio}
\label{sec:biblio}
\end{document}